\documentclass[a4paper,11pt,twoside,reqno]{amsart} 
\usepackage{geometry}
\usepackage[dvipsnames]{xcolor}

 \usepackage[applemac]{inputenc}
 
\usepackage{multicol} 

\usepackage{mathrsfs}
\usepackage{mathabx}

\usepackage{wrapfig}

\usepackage{tabularx}
\newcolumntype{C}[1]{>{\centering\arraybackslash}m{#1}}
\newcolumntype{L}[1]{>{\raggedright\arraybackslash}m{#1}}

\usepackage[all]{xy}

\usepackage{graphicx}
\usepackage{amssymb}

\usepackage{enumerate}

\usepackage{tikz}

\makeatletter
\@namedef{subjclassname@2020}{%
  \textup{2020} Mathematics Subject Classification}
\makeatother

\usepackage{hyperref}
\definecolor{lime}{HTML}{A6CE39}
\DeclareRobustCommand{\orcidicon}{%
	\begin{tikzpicture}
	\draw[lime, fill=lime] (0,0) 
	circle [radius=0.16] 
	node[white] {{\fontfamily{qag}\selectfont \tiny ID}};
	\draw[white, fill=white] (-0.0625,0.095) 
	circle [radius=0.007];
	\end{tikzpicture}
	\hspace{-2mm}
}

\foreach \x in {A, ..., Z}{%
	\expandafter\xdef\csname orcid\x\endcsname{\noexpand\href{https://orcid.org/\csname orcidauthor\x\endcsname}{\noexpand\orcidicon}}
}

 \usepackage[applemac]{inputenc}
\usepackage[T1]{fontenc}

\newtheorem{proposition}{Proposition}[section]
\newtheorem{lemma}[proposition]{Lemma}
\newtheorem{corollaire}[proposition]{Corollary}
\newtheorem{theorem}[proposition]{Theorem}
\newtheorem{theoremv}{Main Theorem}

\theoremstyle{definition}
\newtheorem{definition}[proposition]{Definition}
\newtheorem{example}[proposition]{Example}

\newtheorem{remarque}[proposition]{Remark}

\newcommand{\bl}{\begin{lemma}}
\newcommand{\bp}{\begin{proposition}}
\newcommand{\bt}{\begin{theorem}}
\newcommand{\bc}{\begin{corollaire}}
\newcommand{\be}{\begin{equation}}
\newcommand{\bee}{\begin{equation*}}
\newcommand{\bd}{\begin{definition}}
\newcommand{\bdp}{\begin{definitionproposition}}
\newcommand{\bex}{\begin{example}}
\newcommand{\br}{\begin{remarque}}
\newcommand{\bpr}{\begin{proof}}

\newcommand{\el}{\end{lemma}}
\newcommand{\ep}{\end{proposition}}
\newcommand{\et}{\end{theorem}}
\newcommand{\ec}{\end{corollaire}}
\newcommand{\ee}{\end{equation}}
\newcommand{\eee}{\end{equation*}}
\newcommand{\ed}{\end{definition}}
\newcommand{\edp}{\end{definitionproposition}}
\newcommand{\eex}{\end{example}}
\newcommand{\er}{\end{remarque}}
\newcommand{\epr}{\end{proof}}

\newcommand{\secref}[1]{Section~\ref{#1}}
\newcommand{\subsecref}[1]{Subsection~\ref{#1}}
\newcommand{\thmref}[1]{Theorem~\ref{#1}}
\newcommand{\propref}[1]{Proposition~\ref{#1}}
\newcommand{\lemref}[1]{Lemma~\ref{#1}}

\newcommand{\remref}[1]{Remark~\ref{#1}}
\newcommand{\exemref}[1]{Example~\ref{#1}}

\newcommand{\defref}[1]{Definition~\ref{#1}}

\def\ov{\overline}

\def\menos{\backslash}

\renewcommand{\int}[1]{{\rm int} (#1)}

\def\cT{{\mathcal T}}
\def\cF{{\mathcal F}}
\def\cL{{\mathcal L}}
\def\cS{{\mathcal S}}

\def\cB{{\mathcal B}}

\def\1{{\boldsymbol 1}}

\def\gC{{\mathfrak{C}}}

\def\tc{{\mathtt c}}
\def\tv{{\mathtt v}}

\def\tN{{\widetilde{N}}}
\def\crH{{\mathscr H}}

\def\N {\mathbb N}

\def\R {\mathbb R}

\def\Z {\mathbb Z}

\def\im{{\rm Im\,}}

\def\Hom{{\rm Hom}}

\def\codim{{\rm codim\,}}

\def\rc{{\mathring{\tc}}}

 \def\sd{{\mathrm{sd}}}
\def\sub{{\mathrm{Sub}}}
\def\PL{{\mathrm{PL}}}

\def\comp{{\mathrm{comp} \,}}
    
\def\singf{\mathrm{Sing}^{\cF}}

\def\simp{{\mathrm{Simp}\,}}
\def\PL{{\mathrm{PL}}}
\def\vert{{\mathrm{Vert}\,}}
\def\id{{\mathrm{id}}}

\newcommand{\tres}[3]{{#1}^{^{ #2 }}_{_{#3}}}

\newcommand{\Hiru}[3]{{#1}^{^{#2}}{( #3 )}}
\newcommand{\hiru}[3]{{#1}_{_{#2}}{( #3 )}}
\newcommand{\lau}[4]{{#1}^{^{#2}}_{_{#3}}{\left( #4 \right)}}

\newcommand{\IH}{\mathscr H}

\newcommand{\TO}{\longrightarrow}

\def\tDelta{{\widetilde{\Delta}}}

\def\ttP{{\mathtt P}}

\title{Simplicial intersection homology revisited.}

\date{\today}

\author{David Chataur\orcidC{\!}}
\address{Lamfa\\
Universit\'e de Picardie Jules Verne\\
33, rue Saint-Leu\\
80039 Amiens Cedex~1\\
         France}
\email{David.Chataur@u-picardie.fr}

\author{Martintxo Saralegi-Aranguren\orcidS{\!}
}
\address{Laboratoire de Math{\'e}matiques de Lens\\  
      EA 2462 \\
      Universit\'e d'Artois\\
          62300 Lens Cedex\\
         France}
\email{martin.saraleguiaranguren@univ-artois.fr}

\author{Daniel Tanr\'e\orcidT{\!}
}
\address{D\'epartement de Math{\'e}matiques\\
         UMR 8524 \\
         Universit\'e de Lille\\
         59655 Villeneuve d'Ascq Cedex\\
         France}
\email{Daniel.Tanre@univ-lille.fr}



\makeindex   

\begin{document}

\begin{abstract} 
Intersection homology is defined  for simplicial, singular and PL chains and
 it is well known that the three versions are isomorphic for a full filtered  simplicial complex. 
 In the literature, the isomorphism, between the singular and the simplicial situations of intersection homology,
 uses the PL case as an intermediate.
 Here we show directly that the canonical map between the simplicial and the singular intersection chains complexes
 is a quasi-isomorphism. This is similar to the classical proof for simplicial complexes, with an
 argument  based on the concept of residual complex and not on skeletons.
 
 This parallel between simplicial and singular approaches is also extended to
 the intersection blown-up cohomology that we   introduced  in a previous work.
 In the case of an orientable pseudomanifold, 
 this cohomology owns a Poincaré isomorphism with the intersection homology, 
 for any coefficient ring, thanks to a cap product with a fundamental class.
 So, the blown-up intersection cohomology of a pseudomanifold can be computed 
from a triangulation.
Finally, we introduce a blown-up intersection cohomology for PL spaces and prove that it is isomorphic to the singular one.
\end{abstract} 
\maketitle

The homology of a simplicial complex, $K$,  can be computed indifferently from  the simplices of the triangulation or 
from the singular chains complex of its realization, $|K|$. 
In other words, the map between chain complexes $\hiru C *  K\to  \hiru C {*}{|K|}$ is a quasi-isomorphism.
Classically, the proof (see \cite{MR1867354}) uses an induction on the skeletons $K^{(\ell)}$
of $K$. The crucial point is the existence of   isomorphisms,
$$ \hiru C *{K^{(\ell)}  ,K^{(\ell-1)}}\cong \bigoplus_{\beta\in K, \dim \beta =\ell} \hiru C *{\beta,\partial \beta}
\text{ and }
 \hiru H *{|K^{(\ell)}|  ,|K^{(\ell-1)}|}\cong \bigoplus_{\beta\in K, \dim \beta =\ell} \hiru H *{|\beta|,|\partial \beta|}.$$

\medskip
Before developing the corresponding situation in  intersection homology, 
let us make a brief historical reminder.
In their first paper (\cite{MR572580}) on intersection homology, 
M.~Goresky and R.~MacPherson introduce it for  a pseudomanifold $X$ together with  a fixed PL structure and a parameter $\ov p$,  called  perversity.
 They
 define the complex of  $\ov{p}$-intersection, $\lau {\gC }  {\ov{p}} * X $, as a subcomplex of the complex of PL chains and
 the $\ov{p}$-intersection homology as the homology of $\lau {\gC }  {\ov{p}} * X $.
 Later, in an appendix to  \cite{MR833195},  they define a complex of
 $\ov{p}$-intersection, $\lau C  {\ov{p}} * K$, as a subcomplex of the simplicial chains of a filtered simplicial complex $K$.
 If $K$ is a full admissible triangulation of a PL space X, they  prove that the inclusion
 $\lau C  {\ov{p}} * K\to \lau {\gC }  {\ov{p}} * X $ induces an isomorphism in homology.
 They do that with a nice construction of a left inverse to this  inclusion.
 Let us specify that, without the ``full'' hypothesis on $K$, the simplicial and PL intersection homologies 
 may not be isomorphic, as an explicit example shows in \cite[Appendix]{MR833195}.
 The barycentric subdivision of a simplicial complex being full, one can always find such triangulation of $X$.

 \medskip
 A third step is to consider the topological realization, $|K|$, of $K$ and define a complex of 
$\ov{p}$-intersection, $\lau C  {\ov{p}} * {|K|}$, as a subcomplex of the singular chains on $|K|$. 
This  has been done by King (\cite{MR800845}) who proves the existence of an isomorphism between the
singular and the PL intersection homologies, in the case of CS sets. 
These relationships between the three points of view are developed in a thorough way 
by G.~Friedman in  \cite[Sections 3.3 and  5.4]{LibroGreg}.
Let us emphasize that, unlike that of the classical case, the proof is not a direct comparison 
between the simplicial and the singular points of view:
the PL case is used as intermediate between them and a structure of CS set is required.

The first objective of this work is to obtain an isomorphism between singular and simplicial
intersection homologies, in a direct way, from a chain map between the corresponding complexes,
and without any restriction to  CS sets.
Let us start with a filtered simplicial complex $K$,
that is, a simplicial complex endowed with a filtration,
$$
K = K_{n} \supset K_{n-1} \supset K_{n-2} \supset \ldots \supset 
K_{0} \supset K_{-1} = \emptyset ,  
 $$
where each $K_i$ is a subcomplex.
Unfortunately we \emph{do not} have the expected formula
\begin{equation}\label{bb}
\lau C {\ov p} *{K ^{(\ell)},K^{(\ell-1)}}\cong
\bigoplus_{ \beta \in K, \  \dim \beta =\ell} \lau C {\ov p}*{\beta,\partial \beta},
\end{equation}
as shows the example in  Subsection \ref{Break}.
To overcome this defect, we introduce for any filtered simplicial complex, $K$, a pair of integers, 
called \emph{complexity}  (see \defref{complexity}) and denoted $\comp K$.
With the lexicographic order, the complexity provides the opportunity for reasoning by induction.
As for an alternative decomposition to \eqref{bb}, we introduce a subcomplex of K, called the
 \emph{residual complex}, denoted  $\cL(K)$ and already present in \cite{MR1346255}.
 The desired decomposition comes from a particular class 
 $\cB(K)$ of simplices of $K$, called \emph{clot}, see \defref{clot}.
 By noting $L_{\beta}$ the link of a clot $\beta$, we get in \propref{SimplRel},  an isomorphism
 \begin{equation}\label{SimplRelformIntro}
\bigoplus_{\beta \in \cB(K)}   \lau C  {\ov p} *{ \beta * L_\beta,  \partial \beta * L_\beta}  
\xrightarrow{\cong}
\lau C {\ov p} *{ K,\cL(K)}.
\end{equation}
 From it, we can adapt the proof of the classical case and prove that the canonical map between
 the simplicial and the singular intersection complexes, is a quasi-isomorphism.

 \medskip
In fact, we apply this program not only to $\ov{p}$-intersection homology but also to a 
\emph{$\ov{p}$-intersection cohomology} obtained from simplicial blow up, denoted $\tres \crH*{\ov{p}}$
and called \emph{blown-up intersection cohomology.}
For stratified spaces in general, this cohomology is  naturally equipped 
 with cup and cap products for any ring coefficients and with cohomology operations
 (\cite{CST2,2020arXiv200504960C}).
  Regarding duality and pseudomanifolds,  
 since the work (\cite{MR699009}) of Goresky and Siegel, it is known that 
there is no Poincar\'e duality on intersection homology with coefficients in any commutative ring. 
However, in \cite{CST7}, we prove that the cap product with the fundamental class of an oriented, compact, 
$n$-dimensional pseudomanifold induces an isomorphism between 
$\lau \crH * {\ov{p}} X$ and the intersection homology $\lau H {\ov{p}}{n-*}X$.
(Similar versions exist for a paracompact not necessarily compact pseudomanifold, with
 compact supports in cohomology or Borel-Moore homology, \cite{CST8} and \cite{ST1}).
This blown-up cohomology coincides with the cohomology obtained from the linear dual of the chain complexes, when 
$\sf R$ is a field, 
or more generally with an hypothesis of $\sf R$-local torsion free in a part of the intersection homology of links, 
already present in \cite{MR699009}.
Let us also mention that its  sheafification is the Deligne sheaf (\cite{CST8}).
The existence of the isomorphism between simplicial and singular blown-up intersection cohomology is a new result.
In short, we prove in Theorems \ref{DKBis} and \ref{DKTris}:

\smallskip

\begin{theoremv}
Let $K$ be a full filtered complex and $\ov{p}$ be any perversity, then we have isomorphisms:
$$
\xymatrix@R=-2mm@C=1mm{
\lau H  {\ov p} * K  &\cong  &\lau H  {\ov p} * {|K|} &\ \ \  \hbox {and } \ \ \ &
\lau \IH * {\ov p} K  &\cong  &\lau \IH * {\ov p} {|K|}.\\
\hbox{\tiny simplicial} && \hbox{\tiny singular}&&\hbox{\tiny simplicial} && \hbox{\tiny singular}
}
$$
\end{theoremv}

We also define a blown-up cohomology for PL spaces and relate it to the simplicial and the singular blown-up
cohomologies in \thmref{TeoremaPL}.
This theorem contains a second part on intersection homology, with a proof of the existence of an isomorphism between
singular and PL intersection homology, without requiring the existence of a CS set structure. 
This  answers  a question asked by G. Friedman in \cite[Page 234]{LibroGreg}.

\medskip
In the sequel, all the complexes are considered 
with coefficients in an abelian group denoted $\sf R$, which is not explicitly mentioned. 
If $X$ is a topological space, we denote  
$\tc X=X\times [0,1]/X\times \{0\}$ the cone on $X$ and 
$\rc X= X\times [0,1[/X\times \{0\}$ the open cone on $X$. 
The apex of a cone is  $\tv$.

\medskip
Our program is carried out in Sections 1-6 below whose headings are self-explanatory.

\tableofcontents

\section{Filtered simplicial setting}\label{SS}

\begin{quote}
We introduce the definitions and properties necessary for the study of the intersection homology
of a filtered simplicial complex, $K$.
In particular, by allowing proofs by induction, the notions of complexity and residual complex
of $K$ will play an important role in the sequel.
\end{quote}

 A  \emph{simplicial complex} $K$ is a set of simplices
in $\R^p$, $p\leq \infty$, such that
\begin{enumerate}[1.] 
\item if $\sigma,\tau \in K$ and $\sigma \cap \tau \ne \emptyset$ then $\sigma \cap \tau$ is a face of both $\sigma$ and $\tau$,

\item if $\sigma \in K$ and $\tau$ is a face of $\sigma$, written 
$\tau \triangleleft \sigma$,
 then $\tau \in K$,

\item (local finiteness) every vertex of a simplex of $K$ belongs to a finite number of simplices of $K$.
\end{enumerate}

We say that $d \in \N$ is the dimension of $K$, denoted $\dim(K) = d$, 
if  every simplex of $K$ has dimension lower than or equal  to $d$ and $K$ has at least one simplex of dimension $d$. 
If this number does not exist, we say that  $\dim (K)=\infty$. (By convention, if $K=\emptyset$, we write
$\dim K=-\infty$.)
A subcollection $L \subset K$ is a \emph{simplicial subcomplex} of $K$ if it verifies properties 1 and 2.
The union of the simplices of $K$ whose  dimension is smaller than  or equal to a given $\ell \in \N$ is called the \emph{$\ell$-skeleton} of $K$ and denoted $K^{(\ell)}$. 
A \emph{simplicial map} from $K$ to $K'$ is a function from the set of vertices of $K$ to the set of vertices of $K'$
such that the images of the vertices of a simplex of $K$ is a simplex of $K'$.

\medskip
Let $|K|$ be the topological subspace of $\R^p$ formed by the union of the simplices of the simplicial complex $K$.
If $\Delta$ is the standard simplex, we identify $\Delta$ with $|\Delta|$.

\medskip
A \emph{filtered simplicial complex} is a   simplicial complex $K$
endowed with a filtration made up of simplicial subcomplexes,
\begin{equation*}
K = K_{n} \supset K_{n-1} \supset K_{n-2} \supset \ldots \supset 
K_{0} \supset K_{-1} = \emptyset .
\end{equation*}
The integer $k$ is the \emph{(virtual) dimension} of the filtered simplicial complex $K_{k}$, denoted $\dim_v K_k=k$.

\medskip

A \emph{stratum} of $K$ is a non-empty connected component, $S$, of a $|K_i|\menos |K_{i-1}|$. Its dimension is  $\dim_v S =i$ and its codimension is $\codim_v S =n-i$.
The family of strata of $K$ is denoted $\cS_K$  or $\cS$ if there is no ambiguity.
The strata included in $|K|\menos |K_{n-1}|$ are called \emph{regular}, the other ones being called \emph{singular.}

\medskip
We introduce a pair of integers mixing geometrical and virtual dimensions.

\begin{definition}\label{complexity}
The \emph{complexity} of a filtered simplicial complex $K\ne \emptyset$ is  the pair $\comp (K) = (a,b)$ where:
\begin{itemize}
\item[-] $a = \max \{k \in \{0,\ldots,n\} \mid K_{n-k}\ne \emptyset\}$, and
\item[-] $b \in \ov \N = \N \cup \{\infty\}$ is the geometrical dimension of the simplicial complex $K_{n-a}$.
\end{itemize}
By convention, we set $\comp \emptyset =(-\infty,-\infty)$.
\end{definition}
The associated \emph{lexicographic order}  will prove to be the key in the forthcoming proofs by induction.

\medskip
Recall that a subcomplex $L$ of a simplicial complex $K$ is \emph{full} if any simplex of $K$ having all its vertices in $L$
is a simplex of $L$. 
As we will see, the following reinforcement of the notion of filtered simplicial complex is a necessary hypothesis of the main theorem.

\begin{definition}\label{def:full}
A filtered simplicial complex $K$ is said \emph{full} 
if any $K_{\ell}$ of its filtration is full.
\end{definition}

The standard simplex $K=\Delta$ can be endowed with different structures of  filtered simplicial complex.
\begin{itemize}
\item The first filtration is given by the  skeleta: we set $K_{\ell}=\Delta^{(\ell)}$. It is not full.
\item A series of full filtrations is  defined by induction, starting with the choice of one vertex, $K_{0}=\{v\}$.
For the simplicial subcomplex $K_{\ell}$ with $\ell\geq 1$, we choose  an $\ell$-dimensional face of $\Delta$ containing $K_{\ell-1}$.
\end{itemize}
 The fullness property can be recovered with a barycentric subdivision. 
 This is a general fact as the following result shows (see \cite[Remark 2]{MR833195}).

\begin{lemma}\label{bary}
The barycentric subdivision of a filtered simplicial complex is full.
\end{lemma}

\begin{proof}
Consider a filtered simplicial complex
$$
K =K_n  \supset K_{n-1} \supset K_{n-2} \supset \dots \supset K_0 \supset \emptyset,
$$
inducing the filtered simplicial complex
$
K' =K'_n  \supset K'_{n-1} \supset K'_{n-2} \supset \dots \supset K'_0 \supset \emptyset,
$
where the upperscript $'$ indicates the barycentric subdivision. 
Let $\sigma$ be a simplex of $K'$ with $\sigma \cap |K'_i| \ne \emptyset$. We need  to prove that 
 $ \sigma \cap |K'_i| $ is a face of $\sigma$. Notice that   $ \sigma \cap |K'_i| = \sigma \cap |K_i| $.

\smallskip
Let us denote $\widehat{\tau}$ the barycenter of a face $\tau\in K$.
By definition of the barycentric subdivision, any simplex of $K'$ is obtained from barycenters of successive faces. 
More precisely, 
the barycenters of a family of simplices of $K$,  
$\tau_0 \triangleleft \dots \triangleleft \tau_m$, determine a simplex of $K'$,
 $\sigma = [\widehat{\tau_0}, \ldots, \widehat{\tau_m}]$.  
Let us suppose that, for any $\alpha\in K$, we have
\begin{equation}\label{KK'}
\sigma \cap \alpha = \emptyset \hbox{ or }  \sigma \cap \alpha =  [\widehat{\tau_0}, \ldots, \widehat{\tau_\ell}] \hbox{ for some } \ell \in \{0,\dots,m\}.
\end{equation}
As $K_{i}$ is a simplicial complex, the intersection $\tau_m \cap |K_i|$ is a union of faces of $\tau_m$ and we conclude from $\sigma \subset \tau_m$ and \eqref{KK'}:
$$
\sigma \cap |K_i| 
= \sigma \cap \tau_m \cap |K_i| 
= [\widehat{\tau_0}, \ldots, \widehat{\tau_\ell}]
$$
  for some $\ell \in \{0,\dots,m\}$. 
  This implies  the announced conclusion  $\sigma \cap |K_i| \triangleleft \sigma$ and it  remains to prove \eqref{KK'}.
  
  \smallskip
  So, let $\alpha\in K$ and $\sigma = [\widehat{\tau_0}, \ldots, \widehat{\tau_m}]\in K'$ with $\sigma\cap\alpha\neq \emptyset$.
Any point $P \in \sigma$ has the canonical decomposition
$
P =\sum_{j\in J_P} t_j \widehat{\tau_j}
$
 where $t_j \in ]0,1]$ for each $j\in J_P$, $\sum_{j\in J_P} t_j=1$ and $\emptyset \ne J_P \subset \{0, \ldots,m\}$. 
 We define $\ell = \max \{ j \in J_P \mid \ P \in \sigma \cap \alpha\}$, which  belongs to $\{0, \ldots,m\}$ since $\sigma \cap \alpha  \ne \emptyset$.
By construction we have the inclusions $  \sigma \cap \alpha \subset  [\widehat{\tau_0}, \ldots, \widehat{\tau_\ell}]$ and 
$
 [\widehat{\tau_0}, \ldots, \widehat{\tau_\ell}] \subset \sigma$.
By definition, there exists a vertex $Q$ of  $\sigma \cap \alpha$ with $\ell \in J_Q$. 
The condition $\sum_{j\in J_Q} t_j \widehat{\tau_j} \in \alpha$ implies $  \widehat{\tau_\ell} \in \alpha$. 
This gives the second inclusion $[\widehat{\tau_0}, \ldots, \widehat{\tau_\ell}] \subset \alpha$.
 \end{proof}

\subsection{Canonical decomposition of $\sigma\in K$} \label{CaDe}
Let us suppose that the filtered simplicial complex $K$ is full. This implies that any simplex $\sigma \in K$ has the \emph{canonical decomposition}
$$
\sigma = \sigma_0  * \sigma_1 * \dots * \sigma_n,
$$
where 
$
\sigma \cap |K_\ell |  = \sigma_0 * \dots * \sigma_\ell,
$
for each $\ell\in\{0, \ldots,n\}$. 
(We use the convention $\emptyset * E = E$ for any simplicial complex $E$.)
For each $\ell \in\{0, \ldots,n\}$, we have:
$$
\sigma_\ell \ne \emptyset \Longleftrightarrow \sigma \cap |K_\ell |\menos | K_{\ell-1} | \ne \emptyset 
\Longleftrightarrow  \exists S \in \cS_K   \hbox{ with } \dim_v S = \ell \hbox{ and } \sigma \cap S \ne  \emptyset.
$$
By connectedness, 
the stratum $S$ is unique and we denote it  $S =S_\ell$. 
(Notice that $\ell=n$ is equivalent to the regularity of $S_{\ell}$.)
We set 
$$
I_\sigma^K = \{ \ell \in \{0,\ldots,n\} \mid \sigma_\ell \ne \emptyset\}
\quad \text{and}\quad
\cS_{\sigma}^{K}  = \{S \in \cS_K \mid S \cap \sigma \ne \emptyset\}.
$$
 In other words, we have
$I_\sigma^K  = \{\dim_v  S \mid S \in \cS_\sigma^K\}$ 
and 
$\cS_\sigma^K  = \{S_\ell  \mid\ell \in I_\sigma^K\}$.
We also denote $I_\sigma^K=I_{\sigma}$ and $\cS_\sigma^K =\cS_{\sigma}$ if there is no ambiguity.

\begin{definition}\label{clot}
Let $K$ be a filtered simplicial complex of complexity $(a,b)$.
A simplex $\sigma \in K$ is a \emph{clot} if 
 $\sigma = \sigma_{n-a}$ and
 $\dim \sigma =b$,
where $\comp (K) =(a,b)$.
The family of clots of $K$ is denoted  $\cB(K)$.
\end{definition}
So, a clot is a simplex $\sigma \in K$ living in $K_{n-a}$ with 
maximal geometrical dimension; i.e.,  $\dim \sigma = \dim K_{n-a}=b$.

\subsection{Induced filtered simplicial complexes} \label{IFC}

Let $
K $ be a filtered simplicial complex. Given a  simplicial subcomplex $L \subset K$
 the induced filtration $L_{i}=L\cap K_{i}$ defines
 a filtered simplicial complex structure on $L$.
 Notice that $\comp (L) \leq \comp (K)$.

Since $L_i \subset K_i$ for any $i\in\{0, \ldots,n\}$,  for any stratum $T \in \cS_L$  there exists a unique stratum $S \in \cS_K$ with $T\subset S$. We say that $S$ is the \emph{source} of $T$. Notice that $\dim_v  T = \dim_v S$.

Given $\sigma \in L$, we have a priori two canonical decompositions of $\sigma$: as a simplex of $L$ or
 as a simplex of $K$. Since $\sigma \cap L_k = \sigma \cap L \cap K_k = \sigma \cap K_k$, for each $k \in \{0, \ldots,n\}$, the two decompositions coincide.
In particular $I^K_\sigma= I^L_\sigma$, denoted $I_\sigma$.
Moreover, the family $\cS^K_{\sigma}$ is the family made up of the sources of the strata of $\cS_{\sigma}^L$. 
So the source of $T_\ell$ is $S_\ell$ for each $\ell \in I_\sigma$.

\subsection{Links and joins} \label{links}
 We introduce two  geometrical constructions associated to a  filtered  simplicial complex $K$, of complexity $\comp (K) = (a,b)$.

\medskip
The \emph{link} of a simplex $\beta$ of $ K$  is the  simplicial complex
\[
L_{\beta}= \{\sigma\in K \mid  \beta * \sigma \in K \}.
\]
Recall that $L_\beta$ inherits from $K$ a filtered simplicial complex structure.
 If the simplex $\beta$ is a clot and the filtered simplicial complex  $K$ is full, then
$\left(L_\beta\right)_{n-k}=\emptyset$ for $k\in \{ 0, \ldots, a \}$. In particular, $\comp (L_\beta) < \comp (K)$.
In this case, the \emph{join} $\beta * L_\beta$ of a clot $\beta$ inherits from
$ K$ the following structure of filtered simplicial complex:
\begin{equation}\label{DKBiss}
(\beta *L_\beta)_{i} =
\left\{
\begin{array}{cl}
\beta * (L_\beta)_{i} & \hbox{if }  n-a <i \leq n, \\
 \beta & \hbox{if } i = n-a,\\
\emptyset & \hbox{if } 0\leq i < n-a. \\
  \end{array}
  \right.
  \end{equation}
The family  $\cS_{\beta * L_\beta}$of strata of the join $\beta * L_\beta$  comes from the filtration \eqref{DKBiss}.
\subsection{Residual complex} \label{Res}

\begin{wrapfigure}{r}{0.4\textwidth} 
    \centering
    \includegraphics[width=0.4\textwidth]{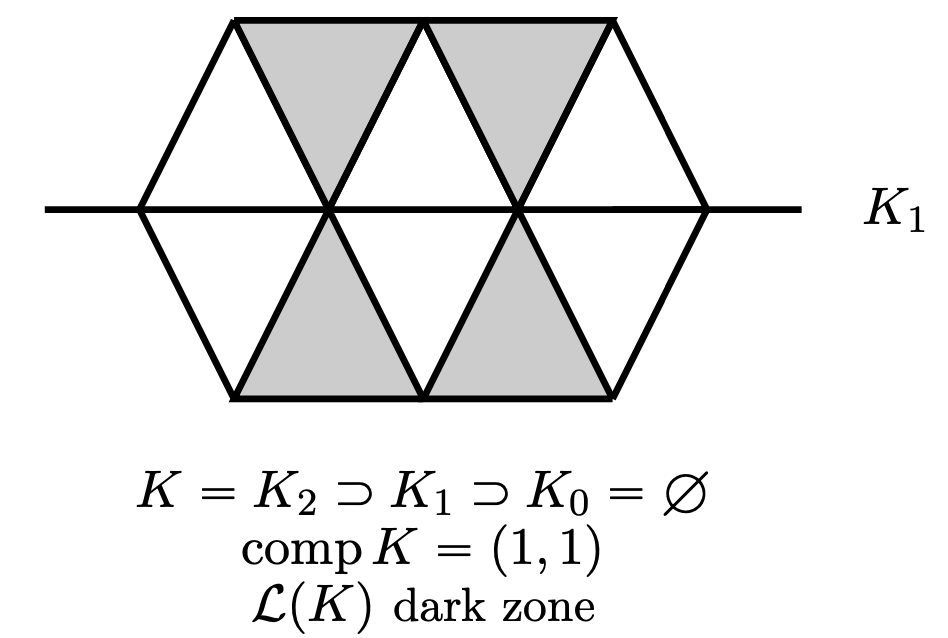}
\end{wrapfigure}

The \emph{residual complex} $\cL(K)$ of a filtered simplicial complex $K\ne \emptyset$ with $\comp K =(a,b)$ is the simplicial subcomplex defined by\\
\phantom . \hfill 
$
\cL(K) = \{ \sigma \in K \mid \dim (\sigma \cap | K_{n-a} | )<b\}.
$
\hfill
\phantom .

Given a simplicial subcomplex, $L\subset K$, we necessarily have
\phantom . \hfill 
$\cL(L)\subset \cL(K).$
\hfill
\phantom .

The equality  $\cL(K) =K$ occures if $b = \infty$. When $b$  is finite, we always have $\comp \cL(K) < \comp K$.
For $a=0$ we get $\cL(K)=K^{(b-1)}$. 

If $K$ is full,
any simplex $\sigma\in K$  has a canonical decomposition,
$\sigma=\sigma_{n-a}*\sigma_{n-a+1}\ast \dots\ast \sigma_{n}$, from which  we deduce,
$$
\sigma\in\cL(K) \Longleftrightarrow \dim \sigma_{n-a}<b\Longleftrightarrow \sigma_{n-a}\text{ is not a clot.}
$$
This gives the following decomposition
\begin{equation}\label{unionK}
K = \cL(K) \cup   \displaystyle \bigcup_{\beta \in \cB(K)} \beta * L_\beta,
\end{equation}
 and
\begin{equation}\label{interLK}
\cL(K) \cap   \displaystyle \bigcup_{\beta \in \cB(K)} \beta * L_\beta =  \bigcup_{\beta \in \cB(K)} \partial \beta * L_\beta
=  \bigcup_{\beta \in \cB(K)} \cL(\beta * L_\beta).
\end{equation}
The induced filtration on the join $\beta\ast L_{\beta}$ is that described in  \eqref{DKBiss}.
\subsection{Perversities}
A \emph {perversity} on a filtered simplicial complex $K$ is a map
$ \ov {p} \colon \cS_K \to \ov \Z =\Z \cup \{\pm\infty\}$ taking the value ~ 0 on the regular strata.
The couple $(K,\ov p)$ is a \emph{perverse filtered simplicial complex}. 
The \emph{dual perversity} $D\ov p$ is defined by $D\ov p(S) = \codim_v S -2 - \ov p(S)$ for any singular stratum $S \in \cS_K$.
For each $k\in \ov \Z$, we denote  $\ov k$ the perversity taking the value $k$ on each singular stratum.

Given a simplicial subcomplex $L \subset K$ and a perversity $\ov p$ on $K$, 
we also denote  $\ov p$ the  perversity defined on the  induced 
filtered simplicial complex $L$  by 
\begin{equation}\label{SE}
\ov p ( T) = \ov p(S),
\end{equation}
where $T\in \cS_L$ and $S\in \cS_K$ is the source of $T$. 
In general, in the rest of the text, the perversities considered on L are perversities induced from a perversity on K.

\section{Simplicial intersection homology}\label{sih}

 \begin{quote}
We present the simplicial version of the intersection homology associated to a filtered simplicial complex, as it appears
in the Appendix of \cite{MR833195}, or  with a more detailed wording in \cite[Section 3.2]{LibroGreg}.
We also develop some generic examples, such as the intersection homology of the pair $(K ,\cL(K))$ (cf. \propref{SimplRel})
or the intersection homology of a join (cf. \propref{CalculoJoin}).
\end{quote}

In this work, the   simplicial complexes  are not supposed to be oriented, but we use oriented  simplices.
An \emph{oriented simplex} of a simplicial complex $K$ is a  simplex of $K$ with an equivalence class of orderings 
of its vertices, where two orderings are equivalent if they differ by an even permutation. 
The \emph{simplicial chain module} $\hiru C*K$ is the module generated by the oriented simplices of $K$.
As a simplex is determined by its vertices we denote it by $[v_{0},\dots,v_{\ell}]$. 
So, two writings  corresponding to an even permutation of the vertices are identified. 
For an odd permutation, $\nu$, we set
 $[v_{0},\dots,v_{\ell}]= - [v_{\nu(0)},\dots,v_{\nu(\ell)}]$, which makes sense at the level of chains.

\medskip
Let $(K,\ov p)$ be a perverse filtered simplicial complex. 
A simplex $\sigma$ of $K$ is $\ov{p}$-\emph{allowable} if, for each  stratum $S\in \cS_K$, we have
\begin{equation}\label{defsimpl}
 \|\sigma\|_{S} = \dim (\sigma \cap S)  \leq \dim \sigma - \codim_v S +\ov p(S). 
\end{equation}
This condition is always satisfied when the stratum $S$ is regular. Thus, the 
\emph{$\ov{p}$-allowability condition} is equivalent to the inequality,
\begin{equation}\label{defsimpl2}
 \|\sigma\|_{S} = \dim (\sigma \cap S)   
 \leq \dim \sigma - 2 - D\ov p(S),
\end{equation}
for each \emph{singular} stratum $S\in \cS_{K}$.
Notice that $\sigma \cap S$ is a union of open faces of $\sigma$, so the dimension of $\sigma \cap S$ makes sense.
A chain $c \in \hiru C * K$ is a $\ov{p}$-\emph{allowable chain} if any simplex with a non-zero coefficient in $c$ is 
$\ov p$-allowable. It is a $\ov{p}$-\emph{intersection chain} if $c$ and $\partial c$ are $\ov{p}$-allowable chains.
The complex of $\ov{p}$-intersection chains is denoted $\lau C {\ov{p}}* K$.
The associated homology is the \emph{simplicial $\ov{p}$-intersection  homology} $\lau H {\ov{p}}* K$, or simply
the  \emph{simplicial intersection  homology} if there is no ambiguity.

\medskip
If $K_{n-1}=\emptyset$,  then $\lau C {\ov{p}}* K$ is the usual simplicial chain complex $\hiru C* K$.

\medskip
If $K$ is full, for each $\ell \in I_\sigma$ the number $\|\sigma\|_{S_\ell}$
is equal to $\|\sigma\|_\ell = \dim (\sigma_0* \dots * \sigma_\ell)$ and the allowability condition \eqref{defsimpl} becomes,
for each $\ell \in I_\sigma$ (i.e., $\sigma_{\ell}\neq \emptyset$),
\begin{equation}\label{defsimplBi}
 \|\sigma\|_{\ell}    \leq \dim \sigma - (n-\ell) +\ov p(S_\ell). 
\end{equation}

There is an important difference between the complex of simplicial chains and  that of $\ov{p}$-intersection chains. 
The first one is a free module over the family of simplices.  
 This is not the case in the second context since there are  $\ov{p}$-allowable simplices which are not 
 $\ov{p}$-intersection chains. 
 
\subsection{Relative intersection homology}\label{RIH} 
Let $(K,\ov p)$ be a perverse filtered simplicial complex
and  $L \subset K$ be a simplicial subcomplex, endowed with the induced filtration and perversity. 
The allowability condition of a simplex $\sigma \in L$ can be understood in $L$ itself or in $K$. 
Both points of view are equivalent. Let us see that.
The simplex $\sigma$ is $\ov{p}$-allowable in  $K$ if, and only if,
$$
\|\sigma\|_\ell \leq \dim\sigma -(n-\ell) + \ov p(S_\ell),
$$
for each $\ell \in I^K_\sigma$.
The simplex $\sigma$ is $\ov{p}$-allowable in  $L$ if, and only if,
$$
\|\sigma\|_\ell \leq \dim\sigma -(n-\ell) + \ov p(T_\ell),
$$
for each $\ell \in I^L_\sigma$.
These conditions are equivalent since the two decompositions of $\sigma$ are the same,  $I^K_\sigma=I^L_\sigma$,
and $\ov p(S_\ell) =  \ov p(T_\ell)$. %
The natural inclusion $L \hookrightarrow K$ gives the following exact sequence
\begin{equation}\label{sucesasta1}
\xymatrix@R=.1cm{ 
 0 \ar[r] 
 & 
\ar@{^{(}->}[r]
 \lau C   {\ov p} * L 
 & 
   \lau C   {\ov p} * K  \ar[r] 
   &  
\displaystyle \frac{  \lau C   {\ov p} * K } {\lau C   {\ov p} * L  }=   \lau C {\ov p}*{K,L}  \ar[r]
 & 0,
 }
 \end{equation}
  defining the \emph{relative simplicial intersection complex} $ \lau C {\ov p}*{K,L}$.
Its homology  is the \emph{relative  simplicial intersection homology} denoted  
 $\lau H{\ov{p}}* {K,L}$.

\subsection{Example }\label{Break}

 \begin{wrapfigure}{r}{0.3\textwidth} 
    \centering
    \includegraphics[width=0.3\textwidth]{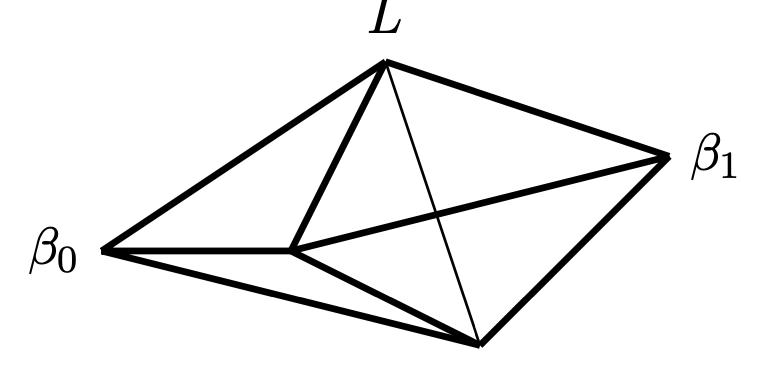}
\end{wrapfigure}
The determination of the relative complex $\lau C {\ov p} *{ K^{(\ell)} , K^{(\ell-1)} }$ does not go through the family of $\ell$-simplices of $K $ as in the classical case. 
Let us see an example.

As simplicial complex $K$, we consider the suspension of a triangle $L$. 
We denote $\beta_{0}$, $\beta_{1}$ the apexes of the suspension.
The 2-skeleton of $K$ is $K$ and its 1-skeleton  is the union of the edges.
This simplicial complex $K$ is endowed with the filtration
$$
\emptyset = K_{-1} \subset \{\beta_0,\beta_1\} = K_0 = K_1 \subset K_2=K.
$$
A straightforward calculation gives, for the $\ov 0$-perversity, $\lau H {\ov 0} j {K^{(2)}, K^{(1)}} =0$, for $j\neq 2$ and
$
 \lau H {\ov 0} 2 {K^{(2)}, K^{(1)}} = {\sf R} \oplus { \sf R }.
$
On the other hand, we have  $\bigoplus_{\sigma\in K, \dim \sigma =2} \lau H {\ov 0} 2 {\sigma ,  \partial \sigma} = 0$.
Therefore the formula \eqref{bb} cannot be true. 

For having a decomposition of the relative intersection homology, we must replace the skeleton by the residual complex.
Let us notice that  $\cL(K) = L$, which is also the link of $\beta_0, \beta_1$. Here, we do have the 
following decomposition, for any $j$,
$$
\lau H {\ov 0} j {K,\cL(K)} = \lau H {\ov 0} j {K,L}   =\bigoplus_{i=0,1} \lau H {\ov 0} j {\beta_i*L,L} 
=\bigoplus_{i=0,1} \lau H {\ov 0} j {\beta_i*L,\cL(\beta_i*L)}.
$$

This decomposition exists in the general case of a full simplicial complex, as we prove in \propref{SimplRel}.
The  proof proceeds by induction on the complexity.
 Firstly, we need the following Lemma.

\begin{lemma}\label{allowsimp}
Let  $(K,\ov p)$ be a perverse  full  filtered simplicial complex. 
Consider a face $\alpha$ of a clot $\beta \in \cB(K)$ and a simplex $\varepsilon$ of the link ${L_\beta}$. 
If the stratum   $Q\in\cS_K$ containing $\beta$ is  singular, we have the following equivalence:
\begin{equation}\label{chainsingularbeta}
\ \alpha\ast \varepsilon  \hbox{ is a $\ov{p}$-allowable simplex } 
\Longleftrightarrow 
\left\{
\begin{array}{l}
\dim \varepsilon \geq D\ov p(Q)  +  1\hbox{ and}\\
 \varepsilon \hbox{ is a $\ov{p}$-allowable simplex.  }
\end{array}
\right.
\end{equation}
\end{lemma}
\begin{proof}
We write $q=\dim_v  Q$.
We have
\begin{eqnarray*}
I_{\alpha\ast\varepsilon} &=&
 \{ \ell \in \{0,\ldots, n\} \mid (\alpha\ast\varepsilon) \cap \left( |(\alpha * L_\beta)_\ell |
\menos 
 |(\alpha * L_\beta)_{\ell-1} |
 \right) \ne \emptyset \}
 \\
 &=&
 \{ q \} \cup  \{ \ell \in \{a+1 ,\ldots, n\} \mid  \varepsilon 
 \cap \left( |( L_\beta)_\ell |
\menos 
 |(L_\beta)_{\ell-1} |
 \right) \ne \emptyset \}
 =\{ q\} \cup I_\varepsilon.
 \end{eqnarray*}
 So,
 \begin{eqnarray*}
 &&
\alpha\ast\varepsilon \hbox{ is a $\ov{p}$-allowable  simplex} 
 \Longleftrightarrow
 \left\{
\begin{array}{l}
\|\alpha\ast\varepsilon\|_q \leq \dim (\alpha\ast\varepsilon) - (n-q)+ \ov p(Q) \hbox{ and}\\
\|\alpha\ast\varepsilon\|_\ell \leq \dim (\alpha\ast\varepsilon) -(n-\ell) + \ov p(S_\ell), \ 
\forall \ell \in I_\varepsilon\menos\{n\},
\end{array}
\right.
\\[,2cm]
&& \Longleftrightarrow
 \left\{
\begin{array}{l}
\dim \alpha \leq  \dim  \alpha  +\dim\varepsilon +1 -(n-q)+ \ov p(Q) \hbox{ and}\\
\dim \alpha + \|\varepsilon\|_\ell  +1 \leq \dim  \alpha + 1 + \dim \varepsilon  - (n-\ell)+ \ov p(S_\ell), \ 
\forall \ell \in I_\varepsilon\menos\{n\},
\end{array}
\right.
\\[,2cm]
 &&
  \Longleftrightarrow
 \left\{
\begin{array}{l}
0\leq \dim \varepsilon +1 -(n-q)+ \ov p(Q) \hbox{ and}\\
\|\varepsilon\|_\ell  \leq  \dim \varepsilon  - (n-\ell)+ \ov p(S_\ell), \ \forall \ell \in I_\varepsilon\menos\{n\},
\end{array}
\right.
\Longleftrightarrow
 \left\{
\begin{array}{l}
\dim \varepsilon  \geq  D\ov p(Q) + 1\hbox{ and}\\
 \varepsilon \hbox{ is a $\ov{p}$-allowable simplex.  }
\end{array}
\right.
 \end{eqnarray*}
\end{proof}

In the case of a face $\alpha$ of the clot $\beta$, we have
\begin{equation}\label{chainnull}
\alpha \hbox{ is a $\ov{p}$-allowable  simplex} 
 \Longleftrightarrow
 \left\{
\begin{array}{l}
Q \text{ is singular and } D\ov{p}(Q)+2\leq 0, \text{ or }\\
Q \text{ is regular.}
\end{array}
\right.
\end{equation}

\begin{lemma}\label{decomposition}
Let $(K,\ov{p})$ be a perverse full filtered simplicial complex. Then, we have
$$\lau C  {\ov p} *{K}=\lau C  {\ov p} *{\cL(K)}+ \bigoplus_{\beta \in \cB(K)}  \lau C  {\ov p} *{\beta\ast L_{\beta}}.
$$
\end{lemma}

\begin{proof}
It suffices to prove the inclusion $\subseteq$. 
From \eqref{unionK}, we know that any chain $c\in \lau C  {\ov p} *{K}$ can be written
$c=f+ \sum_{\beta\in \cB(K)} c_{\beta}$
with $f\in \lau C  {} *{\cL(K)}$ and $c_{\beta}\in \lau C  {} *{\beta*L_{\beta}}$.
As  (see \eqref{interLK}) $\cL(K) \cap    \bigcup_{\beta \in \cB(K)} \beta * L_\beta =  \bigcup_{\beta \in \cB(K)} \partial \beta * L_\beta$,
we obtain a unique writing of $c$ as
$$c=f+\sum_{\beta\in \cB(K)} c_{\beta}= f+\sum_{\beta\in \cB(K)} n_{\beta}\ \beta*e_{\beta} + m_{\beta}\ \beta,$$
with 
$n_{\beta}, m_{\beta}\in \sf R$ and $e_{\beta}\in L_{\beta}$.
Since the elements of this decomposition are independent, the $\ov{p}$-allowability of $c$ gives the $\ov{p}$-allowability
 of the chains $f$, $c_{\beta}$, $e_{\beta}$ (see \lemref{allowsimp}).
 Let $Q_{\beta}$ be the stratum containing $\beta$. From \lemref{allowsimp} and \eqref{chainnull}, we 
 also deduce
 $$
 \left\{
 \begin{array}{ll}
 \dim e_{\beta}\geq D\ov{p}(Q_{\beta})+1
 &
 \text{if } n_{\beta}\neq 0, \text{ and}\\[.2cm]
 D\ov{p}(Q_{\beta})+2\leq 0
 &
 \text{if } m_{\beta}\neq 0 \text{ and } Q_{\beta} \text{ singular}.
 \end{array}\right.
 $$
 The boundary $\partial c$ of $c$ is given by
\begin{equation}\label{diffc}
\partial c =  \sum_{\beta \in \cB(K )} n_{\beta}\ \partial \beta* e_\beta +  
\sum_{\beta \in \cB(K )} 
(-1)^{\dim\beta +1}n_{\beta}\ \beta* \partial e_\beta + \sum_{\beta\in \cB(K)}m_{\beta}\ \partial\beta + \partial f.
\end{equation}
(In the case $\dim \beta=0$, the join $\partial \beta*e_{\beta}$ means $e_{\beta}$, and similarly if $\dim e_{\beta}=0$.)
Following a new time  \lemref{allowsimp} and \eqref{chainnull}, we get that each element of 
the first sum in \eqref{diffc} is a $\ov{p}$-allowable chain. Therefore, the chain
$$\sum_{\beta \in \cB(K )} 
(-1)^{\dim\beta +1}n_{\beta}\ \beta* \partial e_\beta + \sum_{\beta\in \cB(K)}m_{\beta}\ \partial\beta + \partial f$$
is $\ov{p}$-allowable and so is every element of this sum. 
In short,  we get  
$f \in \lau C {\ov p}* {\cL(K)}$
and
$c_\beta \in \lau C {\ov p}* {\beta\ast L_\beta  }$, 
for each $ \beta \in \cB(K )$,  which gives the claim.
\end{proof}

\begin{proposition}\label{SimplRel} 
Let  $(K,\ov p)$ be a perverse  full  filtered simplicial complex. The inclusion maps induce an isomorphism of chain complexes,
\begin{equation}\label{SimplRelform}
\bigoplus_{\beta \in \cB(K)}   \lau C  {\ov p} *{ \beta * L_\beta,  \partial \beta * L_\beta}  
\xrightarrow{\cong}
\lau C {\ov p} *{ K,\cL(K)}.
\end{equation}
\end{proposition}

\begin{proof}
Notice that, for each clot $\beta \in \cB(K)$, we have
$\partial \beta * L_\beta=
\cL( \beta * L_\beta)
\subset \cL(K)$.
(The case $\dim\beta=0$ corresponds to 
$L_{\beta}=\cL(\beta*L_{\beta})\subset \cL(K)$.)
Thus the inclusion maps are well defined. 
The result comes from \lemref{decomposition} and \eqref{interLK}.
\end{proof}

\begin{remarque}\label{clas}
Let us consider a filtered simplicial complex $K$ with $\dim K =b< \infty$ and  the trivial filtration $K=K_n\supset \emptyset$. 
So, we have $\comp K = (0,b)$, $K = K^{(b)}$, 
$\beta(K) = \{\beta \in K \mid \dim \beta =b\}$ and $\cL(K) = K^{(b-1)}$. The previous formula \eqref{SimplRelform} becomes 
$$
\lau C {} *{ K,\cL(K)}
\cong
\mathop{\bigoplus_{\beta \in K }}_{\dim \beta =b}    \lau C  {} *{ \beta, \partial \beta}
\cong
\lau C {} *{ K^{(b)},K^{(b-1)}}.
$$
Thus our approach contains the classical formula  \eqref{bb}. 
\end{remarque}

\subsection{Intersection homology of the join}\label{subsec:joinsimplicial}
To make the writing easier, we employ the dual perversity $D\ov{p}$ of $\ov{p}$.
 Notice that the intersection homology $\lau H  {\ov p} 0 {K}$ of a 
filtered simplicial complex can be 0 when there are no regular strata. 
This is why the following statement contains more cases to consider than that of \cite[Proposition 1.49]{CST1} for instance. 
The same phenomenon appears for the calculation of the intersection homology of a cone between  King'statement  \cite[Proposition 5]{MR800845} and Friedman's
 \cite[Theorem 4.2.1]{LibroGreg}.

\begin{proposition}\label{CalculoJoin}
Consider a perverse full filtered simplicial complex $(K,\ov p)$.
Let  $\beta \in \cB(K)$ be a clot such that the stratum $Q \in \cS_K$ containing $\beta$ is a singular stratum.
We have,
$$
\lau H  {\ov p} i {\beta * L_\beta}
=
\left\{
\begin{array}{cl}
\lau H  {\ov p} {i} {  L_\beta}
&
\hspace{1cm}  \hbox{if } \
  i\leq D\ov p(Q),
\\[.2cm]
0&
\hspace{1cm} \hbox{if } 
i>D\ov p(Q),\,i\neq 0,
\\[.2cm]

{\sf R}
&
\hspace{1cm} \hbox{if } 
i=0 
\hbox{ and } D\ov p(Q)<-1,
\\[.2cm]
{\sf R}
&
\hspace{1cm} \hbox{if } 
i=0, \, 
D\ov p(Q)=-1
 \hbox{ and } 
\lau H {\ov p} 0  {L_\beta} \ne 0,
 \\[.2cm]
0
&
\hspace{1cm}  \hbox{if } i=0,\,
D\ov p(Q)= -1  \hbox{ and } \lau H {\ov p} 0 { L_\beta }=  0.
\end{array}
\right.
$$
 The first isomorphism is given by the inclusion $L_\beta \hookrightarrow \beta * L_\beta$. 
 On the third and fourth lines, a generator of $\sf R$ is  a  point in $\beta$ or any 
 $\ov{p}$-allowable point in $L_{\beta}$, respectively.
 \end{proposition}

\begin{proof}
We write  $\beta=\langle v_0, \ldots,v_\ell \rangle$ and $q=\dim_v  Q$.
(Since $Q$ is singular, we  have  $D\ov p(Q) = n-q-2 - \ov p(Q)$.)
Any chain $c \in \hiru C *{\beta * L_\beta}$ can be written as,
\begin{equation}\label{xii}
c = f +  \sum_{\alpha\triangleleft  \beta} c_{\alpha}=
f+ \sum_{\alpha\triangleleft  \beta}  n_{\alpha}\ \alpha * e_\alpha +m_{\alpha}\alpha,
\end{equation}
where $f \in \hiru C *{ L_\beta}$, $n_{\alpha}, m_{\alpha} \in \sf R$ and  $e_\alpha\in L_\beta$, for any $\alpha \triangleleft \beta$.
With these notations, the characterizations \eqref{chainsingularbeta} and \eqref{chainnull} become
\begin{equation}\label{keyjoin}
c \text{ is } \ov{p}\text{-allowable}
\Longleftrightarrow
\left\{
\begin{array}{l}
\dim e_{\alpha}\geq D\ov{p}(Q)+1
\text{ and } e_{\alpha} \text{ is } \ov{p}\text{-allowable,}\quad
\text{if } n_{\alpha} \neq 0,\\
D\ov{p}(Q)+2\leq 0,\quad
 \text{if } m_{\alpha}\neq 0,\\
\text{and } f \text{ is } \ov{p}\text{-allowable.}
\end{array}
\right.
\end{equation}
We determine  $\lau H  {\ov p}  i {\beta * L_\beta}$ by distinguishing the three cases boxed below.

\medskip
\fbox{$\bullet$ $i=0$.}
A chain $c \in \lau C  {\ov p} 0{\beta * L_\beta}$ is of the form
$c=\sum_{k=0}^b n_{k}\ v_{k}+f$ with $f\in \lau C  {\ov p} 0{ L_\beta}$
or
$c=f$ if $D\ov{p}(Q)+1\geq 0$.
We  distinguish 3 cases. The first one, $D\ov{p}(Q)+1>0$, is postponed in the third item, for any degree $i$. 
So, we are left with:
\begin{enumerate}[\hspace*{2cm}]
\item[$\bullet \bullet$]  $D \ov p(Q) +1 = 0$.  Here, $c=f\in \lau C{\ov p} 0  {L_\beta} $. 
Two  $\ov{p}$-allowable points $p', p''\in L_{\beta}$ define the same class in 
$ \lau H{\ov p} 0  {\beta*L_\beta}$, 
since
$p'-p''=\partial(v_{0}*p'-v_{0}*p'')$,
where $v_{0}*p'$ and $v_{0}*p'$ are  $\ov{p}$-allowable (see \eqref{keyjoin}). 
As a point cannot be the boundary of a 1-chain, the vanishing of 
$ \lau H{\ov p} 0  {\beta*L_\beta}$
is equivalent to the non-existence of $\ov{p}$-allowable points in  $\lau C{\ov p} 0  {L_\beta}$ which gives 
the two last lines of the statement.
\item[$\bullet \bullet$] $D\ov p(Q) + 1 < 0$. 
Let us write $f=\sum_{i\in I}m_{i}\ p_{i}$, with $m_{i}\in \sf R$ and $p_{i}\in L_{\beta}^{(0)}$.
From \eqref{keyjoin}, we deduce that $\sum_{i\in I}m_{i}\ v_{0}\ast p_{i}$ is a $\ov{p}$-intersection chain of boundary
$f -\sum_{i\in I}m_{i}\ v_{0}$.
So the map $\iota_{0}\colon \lau H{\ov p} 0  {\beta} \to \lau H{\ov p} 0  {\beta * L_\beta}$,
induced by the canonical inclusion, is surjective.
Recall that $\lau H{\ov p} 0  {\beta}=\lau H{} 0  {\beta}=\sf R$ is generated by $[v_{0}]$.
If $n\, [v_{0}]$, $n\in\sf R$, is sent to zero by $\iota_{0}$, we have 
$n\, v_{0}=\partial \gamma$, with $\gamma \in \lau C{\ov p} 1  {\beta*L_\beta}$. 
Thus the augmentation map $\epsilon$ gives
$n=\epsilon(\partial \gamma)=0$. We deduce
$\lau H{\ov p} 0  {\beta }\cong \lau H{\ov p} 0  {\beta * L_\beta}\cong \sf R$
and the third line of the statement.
\end{enumerate}

\medskip
\fbox{$\bullet$  $i>0$ and $ D\ov  p(Q)  \leq -1 $.} 
Let $ \eta \triangleleft \beta$ be the opposite face to the vertex $v_0$ or $ \eta =\emptyset$ when $\dim \beta =0$. 
We consider a cycle $\gamma \in \lau C {\ov p}{i}{ \beta*L_\beta}$ and we prove that it is a boundary. 
We decompose $\gamma=\gamma_{0}+\gamma_{1}$ where $\gamma_{0}$ is a chain of simplices having $v_{0}$ as vertex and  $\gamma_{1}=\sum_{i} m_{i}\ \gamma'_{i}$, with $\gamma'_{i}\in \eta*L_{\beta}$.
As $\gamma$ is a cycle, a direct calculation shows that $\gamma=\partial (\sum_{i} m_{i}\ v_{0}\ast \gamma'_{i})$,
since $i>0$.
To prove the second line of the statement, 
 we only need to establish that  $v_{0}\ast \gamma'_{i}$ is a $\ov{p}$-allowable simplex.
For that we use \eqref{keyjoin}. As $\gamma'_{i}$ is $\ov{p}$-allowable by hypothesis, 
we have the three different cases:
\begin{itemize}
\item[$\bullet \bullet$] $\gamma'_{i}=\alpha$ where $\alpha \triangleleft \eta$ and $D\ov{p}(Q)+2\leq 0$,
\item[$\bullet \bullet$] $\gamma'_{i}=\alpha*e_{\alpha}$ where $\alpha \triangleleft \eta$ and $e_{\alpha}\in L_{\beta}$
is a $\ov{p}$-allowable simplex,
\item[$\bullet \bullet$] $\gamma'_{i}= e_{\alpha}$ where  $e_{\alpha}\in L_{\beta}$ is a $\ov{p}$-allowable simplex,
\end{itemize}
since $D\ov{p}(Q)+1\leq 0$.
The same analysis  for the simplex $v_{0}*\gamma'_{i}$, again  with \eqref{keyjoin}
and $D\ov{p}(Q)+1\leq 0$, gives the claim.

\medskip
\fbox{$\bullet$ $D\ov  p(Q) \geq 0 $.} 
Let us introduce the following truncation of the complex $\lau C{\ov p} *{L_\beta}$,
$$
\lau {\tau C}{\ov p} i{L_\beta} =
\left\{
\begin{array}{cl}
\lau C {\ov p} i {L_\beta} & \hbox{if }  i > D \ov  p(Q)  +1,\\
\lau C {\ov p} i {L_\beta} \cap \partial^{-1}(0) & \hbox{if }  i = D\ov  p(Q) +1, \\
0 & \hbox{if } i < D\ov  p(Q) +1.
\end{array}
\right.
$$
Notice that 
$
\hiru H {i}{ \lau   {\tau C}{\ov p}  *
{{L_\beta}} } =
\lau H {\ov p} {i} {L_\beta}
$,
if $i\geq   D\ov  p(Q) +1$,
and
$
\hiru H {i}{ \lau   {\tau C}{\ov p}  *{{L_\beta}} } =0$ otherwise.
 Let us consider the map,
 $$
\Psi \colon \hiru C* \beta \otimes \tau\lau C {\ov p} * {L_\beta} \to  \lau C {} {*+1} {\beta * L_\beta},
$$
defined  at the level of simplices by
$
\Psi(\alpha \otimes \gamma ) = (-1)^{\deg \alpha} \alpha * \gamma
$.
(Notice that $\deg \gamma >0$.) 
By abuse of notation, we will also denote $*$ the extension of $\Psi$ to chains. 
We will use it mainly in the case of the boundary of a simplex.
Ce map $\Psi$ verifies the following properties.
 
 $i)$ If $\deg \alpha>0$ we have
\begin{eqnarray*}
\Psi (\partial(c\otimes \gamma )) & =& \Psi (\partial \alpha \otimes \gamma )  +(-1)^{\deg \alpha} \Psi (\alpha \otimes \partial \gamma )
=
 (-1)^{\deg \alpha-1}   (  \partial \alpha * \gamma )+ ( \alpha *\partial \gamma )
 \\
 & =&
 (-1)^{\deg \alpha-1}  (   \partial \alpha * \gamma    +  (-1)^{\deg \alpha-1}   \alpha *\partial \gamma  )
 =
  (-1)^{\deg \alpha-1}   \partial (\alpha*\gamma )
  \\
&  =&
  - \partial \Psi (\alpha \otimes \gamma ).
\end{eqnarray*}

$ii)$ If $\deg \alpha=0$, we have
$$
\Psi (\partial(\alpha \otimes \gamma )) = \Psi (\alpha \otimes \partial \gamma )
=
   \alpha *\partial \gamma 
 =
    \gamma   -\partial (\alpha * \gamma )
 =
 \gamma  - \partial \Psi (\alpha \otimes \gamma ).
$$

 $iii)$ The simplex $\Psi(\alpha\otimes \gamma)$ is $\ov{p}$-allowable, see \eqref{keyjoin}.
 
 \medskip
 From these points, we deduce that the map
$$
\psi \colon \hiru C* \beta \otimes \tau\lau C {\ov p} * {L_\beta} \to  \frac{\lau C {\ov p} {*+1} {\beta * L_\beta}}{\lau C {\ov p} {*+1} {{L_\beta}}},
$$
defined  by
$
\psi(\alpha \otimes \gamma ) = (-1)^{\deg \alpha} [\![ \alpha * \gamma ]\!]
$
is a well-defined chain map of degree 1.
It is clearly a monomorphism. Let us now prove that $\psi$ is surjective.

\smallskip
Let $c\in \lau C {\ov p} i {\beta * L_\beta}$. From \eqref{keyjoin}, we have
$$c=f+ \sum_{\alpha \triangleleft \beta} n_{\alpha}\ \alpha*e_{\alpha}+ \sum_{\alpha \triangleleft \beta}m_{\alpha}\ \alpha,$$
with $f\in C_{*}(L_{\beta})$, $n_{\alpha},\,m_{\alpha}\in \sf R$
and $e_{\alpha}\in L_{\beta}^{(j)}$ with $j\geq D\ov{p}(Q)+1$.
Notice that $\dim e_{\alpha}>0$ for each $n_{\alpha}\neq 0$.
The boundary of $c$ is also $\ov{p}$-allowable, with
$$
\partial c= \partial f +
\sum_{\underset{\dim \alpha=0}{\alpha \triangleleft \beta}} n_{\alpha} e_{\alpha} +
\sum_{\underset{\dim \alpha>0}{\alpha \triangleleft \beta}} n_{\alpha}\ \partial \alpha*e_{\alpha}+
\sum_{\alpha \triangleleft \beta}(-1)^{\deg \alpha -1} n_{\alpha}\ \alpha* \partial e_{\alpha}+
\sum_{\alpha \triangleleft \beta}m_{\alpha}\ \partial \alpha .
$$
From \eqref{keyjoin}, we deduce $f\in \lau C {\ov p} * { L_\beta}$ and $e_{\alpha}\in \tau \lau C {\ov p} * { L_\beta}$,
which implies 
$\sum_{\alpha \triangleleft \beta}(-1)^{\deg \alpha}\ n_{\alpha}\ \alpha \otimes e_{\alpha}\in \lau C {} * { \beta} \otimes \tau 
\lau C {\ov p} * { L_\beta}$. 
This gives the claim since
$$\psi(\sum_{\alpha \triangleleft \beta}(-1)^{\deg \alpha}n_{\alpha}\, \alpha * e_{\alpha})=
[\![\sum_{\alpha \triangleleft \beta}
n_{\alpha}\, \alpha * e_{\alpha}]\!]=
[\![c]\!].
$$

\medskip
Also, the map $\Phi \colon  \tau\lau C {\ov p} * {L_\beta} \to  \hiru C* \beta \otimes \tau\lau C {\ov p} * {L_\beta}$, 
defined by $\Phi(\gamma ) = -v_0 \otimes \gamma $, is a quasi-isomorphism. 
Let us consider the short exact sequence
\begin{equation}\label{equa:short}
\xymatrix{
0 \ar[r] & 
\lau C {\ov p} * {L_\beta} 
 \ar@{^{(}->}[r] & \lau C {\ov p} * {\beta * L_\beta} \ar[r]& \displaystyle
 \frac{\lau C {\ov p} {*} {\beta * L_\beta}}{\lau C {\ov p} {*} {{L_\beta}}}
\ar[r] &
 0.
}
\end{equation}
By using $\Psi $ and $\Phi$, we may replace the homology of the quotient by the homology of the truncation, with a shift of one degree.
Therefore, the long exact sequence associated to \eqref{equa:short} can be written,
$$
\xymatrix{
\dots \ar[r] & 
\lau H {\ov p}{ i+1} {L_\beta} \ar[r] & 
\lau H {\ov p} {i+1} {\beta * L_\beta} \ar[r]&
\hiru H {i}{ \lau   {\tau C}{\ov p}  *{{L_\beta}} }
\ar@^{(->}[r]^-\delta &
\lau H {\ov p} {i}  {L_\beta}\ar[r] &
\dots,
}
$$
where the connecting  map $\delta$ becomes the inclusion. 
Replacing $\hiru H {i}{ \lau   {\tau C}{\ov p}  *{{L_\beta}} }$ by its value, we get
 $\lau H {\ov p} {i} {\beta * L_\beta}  \cong \lau H {\ov p} {i}  {L_\beta}$, if $i\leq  D \ov  p(Q) $,
and $\lau H {\ov p} {i} {\beta * L_\beta} =0$ otherwise.
Moreover, the isomorphisms are induced by the inclusion $L_\beta \hookrightarrow \beta * L_\beta$.
\end{proof}

When $Q$ is a regular stratum, we have $L_\beta =\emptyset$, $\lau H {\ov p} i {\beta *L_\beta} =\lau H {\ov p} i {\beta}  =  \hiru H i\beta =0$ if $i> 0$
and 
$\lau H {\ov p} 0 {\beta *L_\beta}=\lau H {\ov p} 0 {\beta}  =
 \hiru H 0 \beta =\sf R$.

\section{Singular intersection homology}

\begin{quote} 
 The singular version of the intersection homology goes back to King \cite{MR800845}. 
We focus here on the singular  intersection homology of  the realization of a filtered simplicial complex $K$.  
As in the simplicial case of \secref{sih}, we develop  the relative intersection homology of  the pair $(|K| ,| \cL(K)| )$ 
(cf. \propref{SimplRelBis})
and the intersection homology of the realization of a join (cf. \propref{CalculoJoinBis}).
\end{quote}

\begin{definition}\label{def:filteredspace}
A \emph{filtered space} is a Hausdorff topological space $X$ endowed with a filtration by closed subspaces,
$$X=X_{n}\supset X_{n-1}\supset \dots\supset X_{0}\supset X_{-1}=\emptyset.
$$
The integer $n$ is the dimension of $X$. The $i$-strata of $X$ are the   non-empty   connected components of $X_{i}\menos X_{i-1}$. 
The open strata are called \emph{regular},
the other ones being called \emph{singular.} The set of singular strata of $X$ is denoted $\cS_{X}$,
or $\cS$ if there is no ambiguity.
The \emph{dimension}  of a stratum $S \subset X_i\menos X_{i-1}$ is $\dim_v S=i$. 
Its \emph{codimension} is $\codim_v S = n-i$.
\end{definition}

Given a $n$-dimensional filtered simplicial complex $K$, the associated filtration
 \begin{equation*}
|K| = |K_{n} |\supset |K_{n-1}| \supset |K_{n-2}| \supset \ldots \supset 
|K_{0} | \supset K_{-1} = \emptyset ,  
\end{equation*}
defines a $n$-dimensional filtered space. 
By definition, there is a canonical bijection  $\cS_{|K|} \cong \cS_K$.

\subsection{Induced filtered spaces} 
Let $
X = X_{n} \supset X_{n-1} \supset X_{n-2} \supset \ldots \supset 
X_{0} \supset X_{-1} = \emptyset $ be a filtered space. Given a  subset $Y \subset X$,
the induced filtration $Y_{i}=Y\cap X_{i}$ defines a filtered space structure on $Y$.
Since $Y_i \subset X_i$ for any $i\in\{0, \ldots,n\}$, then for any stratum $T \in \cS_Y$ there exists a unique stratum $S \in \cS_X$ with $T\subset S$. 
We say that $S$ is the \emph{source} of $T$. Notice that $\dim_v  T = \dim_v S$.

\subsection{Perversities}
A \emph {perversity} on a filtered space $X$ is a map
$ \ov {p} \colon \cS_X \to \ov \Z $ taking the value ~ 0 on the regular strata.
The couple $(X,\ov p)$ is a \emph{perverse filtered space}. 
Given a subset  $Y \subset X$ and a perversity $\ov{p}$ on $X$, we also denote  $\ov p$ the  perversity defined on the  induced filtered space $Y$  by 
\begin{equation*}\label{SEBis}
\ov p ( T) = \ov p(S),
\end{equation*}
where $T\in \cS_Y$ and $S\in \cS_X$ is the source of $T$.

\subsection{Perverse degree}
For each  stratum  $S \in \cS_X$, the \emph{perverse degree of  a singular simplex $\sigma \colon \Delta \to X$ along   $S$} is 
 $$
 \|\sigma\|_{S}=\left\{
 \begin{array}{cl}
 -\infty &\text{if } S\cap \sigma(\Delta)=\emptyset,\\
\dim \sigma^{-1}(S) &\text{if not.}
  \end{array}\right.
  $$
By definition, the \emph{dimension} of a non-empty subset $A\subset \Delta$ is the smallest $s\in \N$ 
for which the subset  $A $ is included in the $s$-skeleton of $\Delta$.

\subsection{The definition}
We have all the ingredients needed for  the definition of the intersection homology of a filtered space.
(For further information, the reader  should consult the historical definition in \cite{MR800845}, or \cite[Section 3.4]{LibroGreg}, \cite{CST3}.)

\begin{definition}\label{def:allowableMacPherson}
Let   $(X,\ov p)$ be a perverse filtered space. A singular simplex  $\sigma\colon \Delta\to X$, is $\ov{p}$-\emph{allowable} if,
\begin{equation}\label{equa:allowable} 
\|\sigma\|_S \leq \dim \Delta -  \codim_v S +\ov p(S),
\end{equation}
  for each  stratum $S\in \cS_X$. This condition is always satisfied if the stratum $S$ is regular.
  Thus, the $\ov{p}$-allowability condition is equivalent to the inequality 
\begin{equation}\label{equa:allowable2} 
\|\sigma\|_S \leq \dim \Delta -  D\ov{p}(S) -2,
\end{equation}  
  for each \emph{singular} stratum $S\in \cS_{X}$.
A singular chain $c$ is $\ov{p}$-\emph{allowable} if any simplex with a non-zero coefficient in $c$ is 
$\ov{p}$-allowable. It is a \emph{$\ov{p}$-intersection chain} if $c$ and $\partial c$ are $\ov{p}$-allowable chains.
The associated homology is the \emph{singular intersection homology} denoted $\lau H {\ov{p}}* X$.
\end{definition}

When $X_{n-1}=\emptyset$, the complex $\lau C {\ov{p}}* X$ is the usual singular  chain complex $\hiru C* X$.

\subsection{Relative intersection homology} 

Let $(X,\ov p)$ be a perverse filtered space and   $Y \subset X$ be a subspace endowed with the induced filttration and perversity.
The natural inclusion $Y \hookrightarrow X$ gives the following exact sequence,
\begin{equation}\label{sucesasta2}
\xymatrix@R=.1cm{ 
 0 \ar[r] 
 & 
\ar@{^{(}->}[r]
 \lau C   {\ov p} * Y 
 & 
   \lau C   {\ov p} * X  \ar[r] 
   &  
\displaystyle \frac{  \lau C   {\ov p} * X } {\lau C   {\ov p} * Y  }=   \lau C {\ov p}*{X,Y}  \ar[r]
 & 0,
 }
 \end{equation}
defining the relative singular intersection complex $ \lau C {\ov p}*{X,Y}$. Its homology is 
the \emph{relative  singular intersection homology} denoted  
 $\lau H{\ov{p}}* {X,Y}$.

\begin{proposition}\label{SimplRelBis} 
Let  $(K,\ov p)$ be a perverse  full  filtered simplicial complex. The inclusion map induces the isomorphism,
$$
\bigoplus_{\beta \in \cB(K)}   \lau H  {\ov p} *{| \beta * L_\beta|,  | \cL(\beta * L_\beta)| }  
\xrightarrow{\cong}
\lau H {\ov p} *{ |K|,|\cL(K)|}.
$$
\end{proposition}

\noindent \emph{Proof.}
Recall $\cL(\beta * L_\beta)=\partial \beta * L_\beta$. The proof  is divided into two steps.

\medskip

\emph{ Step 1: Thickening}. 
The barycenter of $\beta \in \cB(K )$ is denoted  $b_\beta$. We have the equality
$
 |\beta * L_\beta|= |b_\beta * \partial \beta * L_\beta|
$,
where $\partial \beta=\emptyset$ if $\dim \beta=0$.

   \begin{center}
    \includegraphics[width=0.3\textwidth]{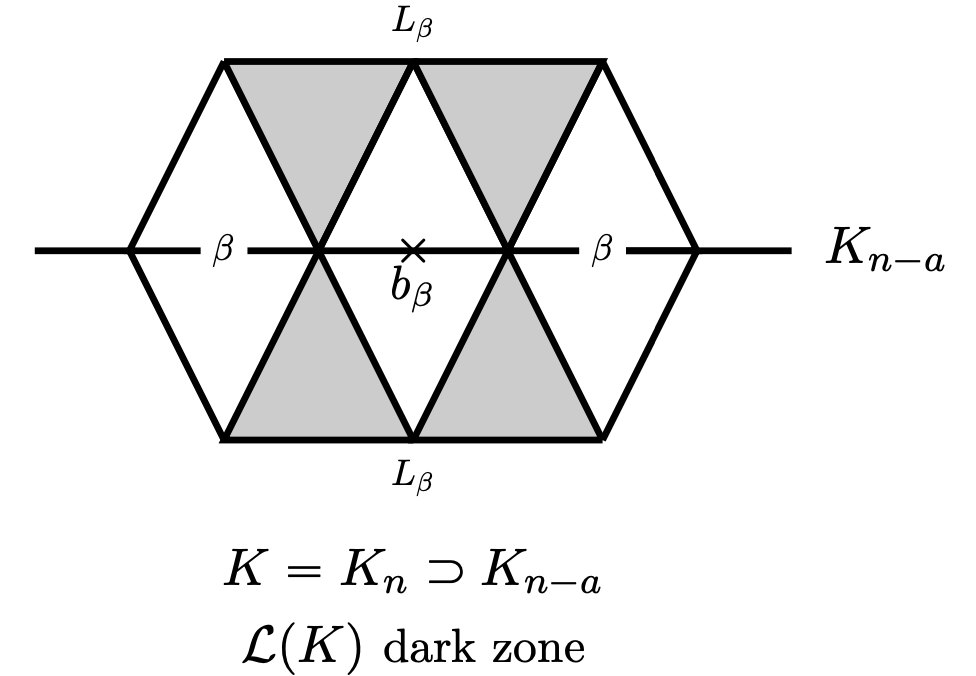}
    \end{center}

\noindent
 So, any point of $x \in |\beta * L_\beta|$ can be written as $x = (1-t)  b_\beta +t a$, where $t\in [0,1]$ and $a \in | \partial \beta * L_\beta|$. 
This writing is unique when $x \ne b_\beta$ (i.e., $t\ne 0$). 
Notice that the assignment $(1-t)  b_\beta +t c \mapsto (c, t)$ 
induces the homeomorphism
\begin{equation}\label{13}
|\beta * L_\beta| \menos \{ b_\beta\} \cong |\partial \beta *L_\beta| \times ]0,1].
\end{equation}

\noindent Under this homeomorphism, the filtration on the subset $|\beta * L_\beta| \menos \{ b_\beta\} $ becomes the product filtration,
with the trivial filtration on $]0,1]$.
Inspired by Subsection~\ref{Res}, we define the open subset
$$
W = |K | \menos \bigcup_{\beta \in \cB(K )} \{b_{\beta}\} = |\cL (K )| \cup  \bigcup_{\beta \in \cB(K )} \left(|\beta * L_\beta |\menos \{b_{\beta}\}\right).
$$
Associated to this open subset, we  define  a map,
$
F \colon W \times [0,1] \to W$
by 
$$
F(x,s) =
\left\{
\begin{array}{ll}
x & \hbox{if } x \in |\cL(K )|,\\
s  x + (1-s)  y &\hbox{if } x= (1-t)  b_\beta +t y \in |\beta * L_\beta | \menos \{b_{\beta}\}.
\end{array}
\right.
$$

Let us specify the restriction of $F$ to $|\beta * L_{\beta}|\menos \{b_{\beta}\}\times [0,1]$, with $\beta \in \cB(K)$.
By definition, we have 
$F((1-t)b_{\beta}+ty,s)=s(1-t)b_{\beta}+(1+st-s)y$.
Composing it with the homeomorphism \eqref{13}, we get a map, still denoted $F$,
from $|\partial \beta *L_{\beta}|\times ]0,1]\times [0,1]$ to $|\partial \beta *L_{\beta}|\times ]0,1]$, defined by
$F(y,t,s)=(y,1+st-s)$. This is the identity on the factor $|\partial \beta *L_{\beta}|$.
Combined with the fact that $F$ is the identity on the factor $\cL(K)$, we conclude that $F$ is a stratified homotopy in the sense
of \cite[Definition 4.1.9]{LibroGreg}.

Let us notice that  $F(-,1)$ is the identity on $W$.
The map $F(-,0)$ is the identity  on $|\cL(K )|$ by construction, and sends $x\notin |\cL(K )|$
on $y\in |\partial \beta\ast L_{\beta}|\subset |\cL(K )|$. Thus, $F(-,0)$ gives a map $\nu\colon W\to |\cL(K )|$.
If we denote  $j\colon |\cL(K ) |\hookrightarrow W$ the inclusion we have  $\nu\circ j = \id _{\cL(K )}$.
On the other hand, $F$ is a stratified homotopy between $F(-,0)=j\circ \nu$ and $F(-,1)=\id_{W}$. Therefore
(\cite[Proposition 4.1.10]{LibroGreg}), the map  $j\circ \nu$ induces the identity map in homology.

In short,  the inclusion $j$ induces an  isomorphism $\lau H  {\ov p} *{ W}\cong \lau H  {\ov p} *{|\cL(K )|}$ 
and therefore an  isomorphism
$$
\lau H  {\ov p} * {|K |, W} 
\cong
\lau H  {\ov p} *{|K |,|\cL (K )|}.
$$

\medskip

\emph{Step 2: Excision.}  By excision (see \cite[Corollary 4.4.18]{LibroGreg}),  we get an isomorphism,
$$
\lau H  {\ov p} *{|K |, W} 
\cong
 \lau H  {\ov p} *{|K | \menos |\cL(K )|, W \menos |\cL(K )|}.
$$
From \subsecref{Res},  we have the disjoint unions
\begin{equation}\label{union}
|K | \menos  |\cL(K )|  = \bigsqcup_{\beta \in \cB(K )} |\beta * L_\beta| \menos |\partial \beta *L_\beta| \ \ \ \hbox{ and } \ \ \ 
W\menos |\cL(K)|  = \bigsqcup_{\beta \in \cB(K )} |\beta * L_\beta| \menos (\{ b_\beta \} \cup |\partial \beta *L_\beta|) .
\end{equation}
This implies
$$
 \lau H  {\ov p} *{|K | \menos |\cL(K )|, W \menos |\cL(K )|} 
\cong
 \bigoplus_{\beta \in \cB(K )} \lau H  {\ov p} *{|\beta * L_\beta| \menos |\partial \beta *L_\beta|, 
 |\beta * L_\beta| \menos (\{ b_\beta \} \cup |\partial \beta *L_\beta| }. 
$$%
An excision relatively to the closed subset $|\partial\beta * L_\beta|$ gives
$$
\lau H  {\ov p}* {|\beta * L_\beta| \menos |\partial \beta *L_\beta|, 
 |\beta * L_\beta| \menos (\{ b_\beta \} \cup |\partial \beta *L_\beta \|)}
  \cong \lau H {\ov p} *{|\beta * L_\beta|,|\beta *L_\beta| \menos  \{ b_\beta \} }. 
$$
Finally,  
from \eqref{13}, we may replace 
$|\beta * L_\beta| \menos \{ b_\beta\}$ by $|\partial \beta *L_\beta| \times ]0,1]$ and obtain
$$
\lau H  {\ov p} * {|\beta * L_\beta|,|\beta *L_\beta| \menos  \{ b_\beta \} } \cong
\lau H  {\ov p}* {|\beta * L_\beta|,|\partial \beta *L_\beta| },
$$
which gives the claim.
\hfill $\clubsuit$

\subsection{Intersection homology of the join} 
As in Subsection~\ref{subsec:joinsimplicial}, we use the dual perversity. The following determination meets the same pattern as in \propref{CalculoJoin}.

\begin{proposition}\label{CalculoJoinBis}
Consider a perverse full filtered simplicial complex $(K,\ov p)$.
Let  $\beta \in \cB(K)$ be a clot such that the stratum $Q \in \cS_K$ containing $\beta$ is a singular stratum. 
We have
$$
\lau H  {\ov p} i {|\beta * L_\beta|}
=
\left\{
\begin{array}{cl}
\lau H  {\ov p} {i} { | L_\beta|}
&
\hspace{1cm}  \hbox{if } \
  i\leq D\ov p(Q),
\\[.2cm]
0&
\hspace{1cm} \hbox{if } 
i >D\ov p(Q),\,i\neq 0,
\\[.2cm]

{\sf R}
&
\hspace{1cm} \hbox{if } 
i=0 
\hbox{ and } D\ov p(Q)<-1,
\\[.2cm]
{\sf R}
&
\hspace{1cm} \hbox{if } 
i=0, \, 
D\ov p(Q)=-1
 \hbox{ and } 
\lau H {\ov p} 0  {|L_\beta|} \ne 0,
 \\[.2cm]
0
&
\hspace{1cm}  \hbox{if } i=0,\,
D\ov p(Q)= -1  \hbox{ and } \lau H {\ov p} 0 { |L_\beta |}=  0.
\end{array}
\right.
$$
The first isomorphism is given by the inclusion $|L_\beta| \hookrightarrow |\beta * L_\beta|$.
On the third  and fourth lines, a generator of $\sf R$ is any point in $\beta$ or any
$\ov{p}$-allowable point in $L_{\beta}$, respectively.
\end{proposition}

\begin{proof} 
We write $q=\dim_v  Q$.
Since $Q$ is singular, we  have  $D\ov p(Q) = n-q-2 - \ov p(Q)$.
The case $\dim \beta  =0$ is given by the classical cone formula, see \cite[Theorem 4.2.1]{LibroGreg}. 
When $\dim \beta  >0$ we consider $ \eta = \langle v_0 \rangle $ the 0-simplex given by the first vertex $v_0$ of $\beta $. It suffices to prove that the inclusion $ \eta \hookrightarrow \beta $ induces the isomorphism  
$\lau H {\ov p} * {| \eta*L_{\beta}|} \cong \lau H {\ov p} *{|\beta * L_\beta|}$. 
Consider a simplicial homotopy  $F \colon \beta  \times [0,1] \to \beta $ between the identity on $\beta $ and the constant map $\beta  \to  \eta$. 
The realization map, still denoted  $F \colon |\beta * L_\beta|\times [0,1] \to |\beta * L_\beta|$,  is a homotopy between the identity on $|\beta * L_\beta|$ 
and the map $\xi \colon |\beta * L_\beta| \to | \eta *L_\beta|$ given by $v*x \mapsto v_0*x$.

From \cite[Proposition 4.1.10]{LibroGreg}, we get that both maps $F(-,0) ,F(-,1)\colon |\beta * L_\beta| \to |\beta * L_\beta|$ induce the same morphism 
in $\tres H{\ov p}* $-homology. Since $F(-,0)$ is the identity, 
then $(F(-,1))_* \colon \lau H  {\ov p} *{|\beta * L_\beta|} \to \lau H  {\ov p}* {| \eta*L_\beta|}$ is the identity.

Finally, since the identity map  is equal to the composition $ | \eta * L_\beta|  \hookrightarrow | \beta * L_\beta| \xrightarrow{\xi} | \eta *L_\beta|$ and the map $F(-,1)$
to the composition $  |\beta * L_\beta|    \xrightarrow{\xi} | \eta *L_\beta|  \hookrightarrow |\beta * L_\beta|$, 
we conclude that the inclusion $ | \eta * L_\beta|  \hookrightarrow |\beta * L_\beta|$ induces an isomorphism $ \lau H {\ov p} * {| \eta*L_\beta|}  \cong \lau H {\ov p} *{|\beta * L_\beta|} $.
\end{proof}

When $Q$ is a regular stratum we have $L_\beta = \emptyset$, $\lau H {\ov p} i{|\beta *L_\beta|} =\lau H {\ov p} i {|\beta|}  =  \hiru H i{| \beta |}= 0$ if $i>0$
and $\lau H {\ov p} 0{|\beta *L_\beta|} =\lau H {\ov p} 0{|\beta|} =\hiru H 0 {| \beta |}=\sf R$.

\section{Simplicial versus Singular}

\begin{quote}
Let $(K,\ov{p})$ be a perverse filtered simplicial complex.
After a reminder of the canonical inclusion map, 
$
\iota \colon 
\lau C {\ov p} * { K} \to \lau C {\ov p} *{|K|}
$,
we prove in \thmref{DKBis} that the map $\iota$ induces an isomorphism in homology when $K$ is full.
The result is no longer true if we remove the hypothesis ``full,''
as it is pointed out in \cite[Remark 2]{MR833195}.
\end{quote}

Consider a perverse full filtered simplicial complex $(K,\ov p)$. 
With the ``well ordering theorem'' applied inductively to $K_{0}$, $K_{1}\menos K_{0}$, and so on, we can assume that the set of vertices of $K$ is provided
with an order $\leq$ that restricts to a total order on each simplex and verifies
\begin{equation}\label{order2}
v \leq w \hbox{ and } w \in K_k \Rightarrow v \in K_k.
\end{equation}

Let $\sigma \in K$ be an oriented simplex whose vertices are $\{v_0, \ldots,v_i\}$ with $v_0 < v_1 < \dots <v_i$. 
If $\sigma = \langle v_0, \ldots, v_i\rangle$, we say that $\sigma$ is an \emph{ordered simplex}. 
Taking in account the identification made in the definition of oriented simplices, 
we notice that $-\sigma$ is an ordered simplex if $\sigma$ is not and vice-versa. 
So, the chain complex of oriented simplices $\hiru C*K$ is generated by the ordered simplices.
Associated to such simplicial simplex, we have the singular simplex $\iota (\sigma) \colon \Delta \to |K|$ defined by:
\begin{equation}\label{iota}
\iota (\sigma) \left( \sum_{j=0}^i t_j  a_j \right) =  \sum_{j=0}^i t_j v_j,
\end{equation}
 where $\Delta=\langle a_{0},\dots,a_i\rangle$ is the standard simplex.
It is well known that 
$\iota (\sigma) \colon \Delta \to |K|$ is a linear map
and that
$\iota  \colon \hiru C*K \to \hiru C* {|K|}$ is a chain map.

\begin{proposition}\label{def:rho}
Let   $(K,\ov p)$ be a perverse  full filtered simplicial complex of virtual dimension~$n$.
The map $\iota  \colon \lau C{\ov p} * K \to \lau C{\ov p} * {|K|}$, associated to an order verifying \eqref{order2}, 
is a chain map.
\end{proposition}

\begin{proof}
Let $\sigma = \langle v_0, \ldots, v_i\rangle\in K$, with $v_0 < v_1 < \dots <v_i$. It suffices to prove
 $\|\iota (\sigma)\|_S = \|\sigma\|_S$ for each singular stratum $S \in \cS_{|K|}$ verifying  
 $S \cap  \iota (\sigma)(\Delta) =S\cap \sigma (\Delta) \ne \emptyset$.
 As we observed in Subsection~\ref{CaDe},
the canonical decomposition of $\sigma$ is given by
\begin{equation}\label{caraS}
\sigma \cap |K_\ell| = \sigma_0* \dots * \sigma_\ell,
\end{equation}
for each $\ell \in \{0,\ldots,n\}$.
 The sets $\Delta_i = \iota (\sigma)^{-1}(\sigma_i)$  are empty-sets or  faces of $\Delta$. We have
$$
\iota (\sigma)^{-1}(|K_\ell|) = \iota (\sigma)^{-1}(\sigma \cap |K_\ell|) = \iota (\sigma)^{-1}( \sigma_0* \dots * \sigma_\ell) = \Delta_0 * \dots * \Delta_\ell.
$$
If $j =\dim_v  S$, the stratum $S$ is the only $j$-dimensional stratum that meets $\sigma$ and we deduce the claim from
\begin{eqnarray*}
\|\sigma\|_S 
&=& \dim (\sigma \cap S) =
  \dim (\sigma \cap |K_j|\menos |K_{j-1}|) = \dim (\sigma \cap |K_j|)\\
&=& \dim (\sigma_0 * \dots * \sigma_j) =\dim (\Delta_0 * \dots * \Delta_j) = \|\iota (\sigma)\|_S.
\end{eqnarray*}
\end{proof}

The next statement is the existence of  an isomorphism  between the singular and the simplicial intersection homologies, for a full complex. 
This result was proven between the PL and the singular intersection homologies by M. Goresky and R. MacPherson in the Appendix of \cite{MR833195}.
The key point of their proof is taken up in the proof of \propref{prop:Ksub}  below.
The isomorphism between the singular and the PL intersection homologies is set up by H. King in \cite{MR800845} for a CS-set.
 Let us also mention that these two isomorphisms are taken over with detail and comments in 
 \cite[Sections 3.3 and 5.4]{LibroGreg}. 
 So, in the literature, the isomorphism between  the singular and the simplicial intersection homologies
 is established  by using the PL intersection homology as an intermediate.
 Here, the proof comes from a direct comparison between these two homologies and does not need a restriction to CS sets.

\begin{theorem}\label{DKBis}
Let $(K,\ov p)$ be a perverse full filtered simplicial complex, with an order on the set of vertices verifying  \eqref{order2}. 
Then, the associated inclusion,  $\iota  \colon \lau C{\ov p} * K \to \lau C{\ov p} * {|K|}$,  induces the isomorphism
$$\lau H {\ov p} * {K} {\cong}  \lau H  {\ov p} *{|K| }.
$$
\end{theorem}
\begin{proof}
We proceed in two steps.\\
\emph{First Step: we suppose $K=K^{(\ell)}$ for some $\ell \in \N$}.  

Given a subcomplex $L \subset K$ we also have $L=L^{(\ell)}$. If $\comp L =(a_L,b_L)$ then $b_L\leq  \dim L< \infty$ and therefore $\comp \cL(L) < \comp L$ (cf. Subsection~\ref{Res}).
We use an induction on the complexity $(a,b)$ of $K$.
If $a=b=0$ then $K=K_{n}$ is a discrete family of 0-dimensional  simplices. So,  $\lau C {\ov p} * K = \lau C {\ov p} * {|K|}$ and we get the claim.
Let us suppose $(a,b)>(0,0)$ and consider the following commutative diagram defining the relative homology,
$$
\xymatrix{
0 \ar[r] & 
\lau C {\ov p} * {\cL(K)} 
\ar[d]  \ar@{^{(}->}[r]&
\lau C {\ov p} * {K} 
  \ar[r] \ar[d]  &\lau C {\ov p} * {K , \cL(K)}  \ar[d] \ar[r]& 0\\
0 \ar[r] & 
 \lau C   {\ov p} *{|\cL(K)|}
 \ar@{^{(}->}[r]&
\lau C  {\ov p} *{|K |}
   \ar[r]
   &\lau C   {\ov p}* {|K|,|\cL(K)|}
    \ar[r]& 0,
}
$$
where the vertical maps are induced by the inclusion map $\iota$ (cf. \eqref{sucesasta1}, \eqref{sucesasta2}).
From the induction hypothesis,
we know that the left vertical arrow is a quasi-isomorphism.
So, it suffices to prove that the  right vertical arrow is a quasi-isomorphism. 
Using Propositions \ref{SimplRel}  and \ref{SimplRelBis},  this assertion
is equivalent to the fact that, for each clot   $\beta  \in \cB(K )$,
 the following map is a quasi-isomorphism
$$
  \lau C  {\ov p} * {\beta  * L_\beta  ,\cL(\beta  * L_\beta }
 \to
  \lau C {\ov p} * {|\beta  * L_\beta | , |\cL(\beta  * L_\beta |}.
$$
Again we combine the short exact sequences \eqref{sucesasta1} and \eqref{sucesasta2}  in a commutative diagram, 
$$
\xymatrix{
0 \ar[r] &
   \lau C {\ov p} * {\cL(\beta  * L_\beta )}   \ar@{^{(}->}[r]  \ar[d]  &
 \lau C {\ov p} * {\beta  * L_\beta }  \ar[d] \ar[r]  &
  \lau C {\ov p} * { \beta  * L_\beta , \cL(\beta  * L_\beta )}  \ar[d] \ar[r]& 0\\
0 \ar[r] &  
 \lau C  {\ov p} * {|\cL(\beta  * L_\beta |} \ar@{^{(}->}[r]&
\lau C {\ov p} * {|\beta  * L_\beta |} \ar[r] &
\lau C {\ov p} *{|\beta  * L_\beta  |, |\cL(\beta  * L_\beta )|} \ar[r]& 0,
}
$$
where the vertical maps are induced by the inclusion map $\iota$.
As noticed at the beginning of this proof, we know that
$
 \comp \cL(\beta * L_\beta) < \comp (\beta * L_\beta) \leq \comp K
$.
With  the induction hypothesis, the 
left vertical arrow is a quasi-isomorphism. 
It remains to prove that the middle arrow is a quasi-isomorphism.
For doing that, we distinguish two cases.

$\bullet$ \emph{The clot $\beta$ is included in a regular stratum}. Here, we have $L_\beta=\emptyset$ and the middle arrow becomes
$\iota_* \colon 
 \lau C {} * {\beta   }  \to
\lau C {} * {|\beta  |}.$
The claim comes from the classical situation. 

$\bullet$ \emph{The clot $\beta$ is included in a singular stratum}. 
  From Subsection~\ref{links}, we  have $\comp L_\beta < \comp K$. 
From Propositions \ref{CalculoJoin}, \ref{CalculoJoinBis} and   the induction hypothesis,  we deduce 
that the middle vertical arrow is a quasi-isomorphism. 

\medskip
\emph{Second Step: the general case.} We consider the induced map 
$\iota_* \colon \lau H {\ov p} * {K} \TO \lau H  {\ov p} *{|K| }$ and decompose the proof in two points.

 $\bullet$\emph{Claim: $\iota_*$ is an epimorphism.} 
Consider a cycle $c\in \lau C  {\ov p} *{|K| }$. 
The chain $c$ being a finite sum,  there exists an integer $\ell \in \N$ with 
$c\in \lau C  {\ov p} *{ |K^{(\ell)}|}$.
 Applying the first step of the proof, there exist $f \in \lau C  {\ov p} *{|K^{ (\ell) }|}$  and $e \in \lau C  {\ov p} *{K^{ (\ell) }}$ with $\partial e =0$ and  $\iota(e) = c + \partial f$. 
 Since  $f \in \lau C  {\ov p} *{|K|}$  and $e \in \lau C  {\ov p} *{K}$, we get  $\iota_*([e]) =[c]$ and  the claim.

$\bullet$ \emph{Claim: $\iota_*$ is a monomorphism.} 
Consider a cycle $e\in \lau C  {\ov p} *{K }$ and a chain $f \in  \lau C  {\ov p} *{|K| }$ with $\iota(e) = \partial f$.
Since $e$  and  $f$ are  finite sums,  there exists an integer $\ell \in \N$ with 
$e\in \lau C  {\ov p} *{ K^{(\ell)}}$ and $f\in \lau C  {\ov p} *{ |K^{(\ell)}|}$.
Applying the first step of the proof, there exists $g \in \lau C  {\ov p} *{K^{ (\ell) }}$ with  $e=  \partial g$. Since  $g\in \lau C  {\ov p} *{K }$, we get $[e]=0$ and the claim.
\end{proof}

The intersection cohomology $\lau H  *{\ov p}  {-}$ of \cite{MR3046315} is defined as the cohomology of  the dual complex $\Hom (\lau C {\ov p}*{-};R)$. 
As the Universal Coefficient Theorem is true in this context (\cite[Theorem 7.14]{LibroGreg}),  the singular and simplicial   intersection cohomologies are isomorphic
for a full filtered simplicial complex.

\section{Blown-up intersection cohomologies}

 \begin{quote} 
 
 Intersection homology and cohomology do not verify the Poincaré duality for any ring of coefficient (see \cite{MR699009,MR3046315}). 
 The blown-up intersection cohomology has been introduced in order to recover this property through a cap product with a fundamental class (see \cite{CST7}).
In this section, we extend the paradigm (simplicial, singular) to  the blown-up intersection cohomology.
After their definitions, we show in Propositions \eqref{prop:restr3}, \eqref{RelativoBiss} and \eqref{JoinMain} that the process  used for intersection homology is still valid here.
Finally, we prove the existence of an isomorphism between the simplicial and singular  blown-up intersection cohomologies in \thmref{DKTris},  for a full filtered simplicial complex.
\end{quote}

\begin{definition}\label{def:filteredsimplex}
Let $X$ be a filtered space of dimension~$n$.
 A \emph{filtered singular simplex} of $X$  is a continuous map, $\sigma\colon\Delta\to X$, 
 such that $\sigma^{-1}(X_{i})  $ is a face of $\Delta$, for $i\in \{0,\dots,n\}$, or the empty set.
  The family of these simplices is denoted  $\singf X$.
A filtered singular simplex $\sigma$ is  \emph{regular} if $\sigma^{-1}(X_{n}\menos X_{n-1})\ne \emptyset$. 
The subfamily of regular simplices is denoted  $\singf_+X$.
\end{definition}

  \begin{definition} \label{def:boumsimplex}
  A \emph{filtered simplex} is a decomposition $\Delta = \Delta_0 * \dots * \Delta_n$, where each $\Delta_i$ is a face of the standard simplex $\Delta$  or the empty set. If  $\Delta_n\ne \emptyset$ we say that the simplex is \emph{regular}.
The \emph{blow up} of a regular simplex is the prism
$
\tDelta = \tc \Delta_0 \times \dots \times \tc \Delta_{n-1} \times \Delta_n.
$
 \end{definition}

The domain of a filtered singular simplex, $\sigma\colon \Delta\to X$,  inherits a decomposition
$$\Delta_\sigma=\Delta_{0}\ast\Delta_{1}\ast\dots\ast\Delta_{n},$$
with $\sigma^{-1}(X_{i})=\Delta_{0}\ast\dots\ast\Delta_{i}$, and $\Delta_{i}$ possibly empty,
for each~$i \in \{0, \dots, n\}$.
We call it the  \emph{$\sigma$-decomposition of $\Delta$.}  
This simplex $\sigma$  is said regular if   $\Delta_n\ne \emptyset$. Notice that each filtered singular simplex $\sigma \colon \Delta \to X$ (resp. regular) induces a filtered simplex (resp. regular) on $\Delta$.
The family $\singf X$ is a $\Delta$-set in the sense of Rourke and Sanderson (\cite{MR300281})
for the usual face maps $\delta_{i}$.
The blown-up construction, we describe above,  is the key point in the construction of our cohomology. Let us begin
with the local situation.

  \begin{definition}\label{43}
A \emph{face} of $\tDelta$ is represented by
$
(F,\varepsilon)=(F_{0},\varepsilon_{0})\times \dots \times  
(F_{n},\varepsilon_{n}),
$
with $\varepsilon_{i}\in \{0,1\}$,  $\varepsilon_n=0$ and  $F_{i}$  a face of $\Delta_{i}$ for $i\in\{0,\dots,n\}$, 
or the empty set with $\varepsilon_{i}=1$. 
\end{definition}
More precisely, we have:
\begin{itemize}
\item $\varepsilon_{i}=0$ and $F_{i} \triangleleft  \Delta_i$, that is, $(F_i,0)= F_i$ is a face of $\Delta_i$, or
\item $\varepsilon_{i}=1$ and  $F_{i} \triangleleft  \Delta_i$, that is, $(F_i,1)= \tc F_i$ is the cone of a face  of $\Delta_i$, or
\item  $\varepsilon_{i}=1$ and $F_{i}=\emptyset$, that is, $(\emptyset,1)$ is the apex of the cone $\tc \Delta_i$, 
called the \emph{virtual apex}. 
\end{itemize}
Recall  $\varepsilon_n=0$. 
We also set
$$
|(F,\varepsilon)|_{>j} = \dim (F_{j+1},\varepsilon_{j+1}) + \dots + \dim (F_{n},\varepsilon_{n}).
$$
Let us denote  $\Hiru N*\Delta$ the complex of simplicial cochains defined on the standard simplex $ \Delta$, with coefficients in  $\sf R$. 

\begin{definition}
The \emph{blown-up complex} of a regular simplex  $\Delta = \Delta_0* \dots * \Delta_n$ is the tensor product
$$
\tN^*(\Delta) = N^*(\tc \Delta_0) \otimes \dots \otimes N^*(\tc \Delta_{n-1}) \otimes N^*(\Delta_n).
$$
\end{definition}

Let $ \1_{(F_{i}, \varepsilon_ {i})} $ be the cochain on
$ \tc \Delta_ {i} $, taking the value $ 1$ on the simplex $(F_ {i} ,
\varepsilon_i)$ and $ 0$ on the other simplices of $ \tc \Delta_i$, for $i\in\{0, \ldots,n-1\}$. 
Similarly for $\1_{(F_n,\varepsilon_n)}$.
A basis of $\tN^*(\Delta)$  is given by the family
$$\1_{(F,\varepsilon)}=\1_{(F_{0},\varepsilon_{0})}\otimes\dots\otimes \1_{(F_{n},\varepsilon_{n})},$$
where $(F,\varepsilon)$ runs over the faces of $\tDelta$. 
Each element of this basis owns an extra degree, coming from the filtration and called perverse degree.

\begin{definition}\label{def:degrepervers}
Let $ \ell \in \{  1, \dots,  n\} $. %
The \emph{$\ell$-perverse degree of the cochain 
$ \1_{(F, \varepsilon)}\in \tN^*(\Delta)$}
is equal to
$$
\|\1_{(F,\varepsilon)}\|_{\ell}=\left\{
\begin{array}{ccl}
-\infty&\text{if}
&
\varepsilon_{n-\ell}=1,\\
|(F,\varepsilon)|_{ > n - \ell}
&\text{if}&
\varepsilon_{n-\ell}=0.
\end{array}\right.$$
The \emph{$\ell$-perverse degree } of
$\omega=\sum_ {\mu} \lambda_{\mu} \, \1_{(F_{\mu} \varepsilon_{\mu})} \in  \Hiru \tN  * {\Delta}$,
with each $ \lambda_ {\mu} \neq 0$, is equal to
 \begin{equation*}
\|\omega\|_{\ell}=\max_{\mu}\|\1_{(F_{\mu},\varepsilon_{\mu})}\|_{\ell}.
\end{equation*}
By convention, we set $ \| 0 \|_{\ell} = - \infty $.
\end{definition}

\begin{remarque}\label{Milagro}
Let us consider a face $(F,\varepsilon)$ of $\tDelta$ with $F_0 =\dots =F_{m-1} =\emptyset$ for some $m \in \{0, \ldots,n-1\}$.  
From the definition, we observe that the  perverse degrees $\|\1_{(F,\varepsilon)}\|_\ell $, for $\ell\in\{1, \ldots,n\}$, do not depend on the face $F_m$.
 \end{remarque}
 
Let $X$ be a filtered space and $\ttP$ be a subset of $\singf X$ 
stable by the face operators. 
The subfamily of its regular elements is denoted $\ttP_{+}$.
In \exemref{exam:simplicialyPL}, we detail two examples of subsets $\ttP$ of interest for this work.

\medskip
Let us define the blown-up cochain complex associated to $\ttP\subseteq \singf X$.
First, to any regular simplex, $\sigma\colon \Delta=\Delta_0\ast\dots\ast \Delta_n\to X$ in $\ttP$, 
we associate the cochain complex defined by 
$$\tres \tN*\sigma=\tN^*(\Delta)=
\Hiru N *{\tc\Delta_0}\otimes\dots\otimes
\Hiru  N *{\tc\Delta_{n-1}} \otimes \Hiru N *{\Delta_n}.$$
A face operator $\delta_{\ell}\colon  \nabla= \nabla_{0}\ast\dots\ast \nabla_{n} \to
 \Delta=\Delta_{0}\ast\dots\ast\Delta_{n}$
is  \emph{regular} if $ \nabla_n\ne \emptyset$. By restriction, 
such $\delta_{\ell}$ induces
 $$
 \delta_\ell^* \colon \Hiru N *{\tc\Delta_0}\otimes\dots\otimes
\Hiru  N *{\tc\Delta_{n-1}} \otimes \Hiru N *{\Delta_n} \to \Hiru N *{\tc \nabla_0}\otimes\dots\otimes
\Hiru  N *{\tc \nabla_{n-1}} \otimes \Hiru N *{ \nabla_n}.
$$

 \begin{definition}\label{def:thomwhitney}
 Let $X$ be a filtered space and $\ttP$ be a sub $\Delta$-set of $\singf X$.
 The \emph{blown-up complex,} $\Hiru \tN {*,\ttP} X$,
is the cochain complex 
 formed by the elements $ \omega $, associating to any $\sigma \in \ttP$ 
an element
 $\omega_{\sigma}\in\tres  \tN*\sigma$,  
so that 
 $
\delta_{\ell}^*(\omega_{\sigma})=\omega_{\sigma\circ\delta_{\ell}},
 $
for any  regular face operator,
 $\delta_{\ell}$, of $\ttP$.
 The differential of $\Hiru \tN {*,\ttP} X$ is defined by
 $(\delta\omega)_{\sigma}=\delta(\omega_{\sigma})$.
  The \emph{perverse degree of $\omega\in\Hiru \tN {*,\ttP} X$   along a singular stratum,} $S \in \cS_X$, is equal to
\begin{equation*}\label{equa:perversstrate}
\|\omega\|_{S}=\sup\left\{ \|\omega_{\sigma}\|_{\codim_v S}\mid \sigma\in \ttP_{+} \text{ with } \im \sigma \cap S\neq \emptyset\right\}.
\end{equation*}
We denote  $\|\omega\|\colon \cS_{X}\to \ov{\Z}$  the map associating to any singular stratum $ S $ of $ K  $ the element $\|\omega\|_{S}$ and 0 to any regular stratum.
 \end{definition}

\begin{definition}\label{def:admissible} 
 Let $\ov{p}$ be a perversity on a filtered space  $X$ and $\ttP$ be a sub $\Delta$-set of $\singf X$.
  A cochain $\omega\in  \Hiru \tN {*,\ttP} X$ is \emph{$\ov{p}$-allowable} if
 $
\|\omega\|\leq \ov p.
 $
A cochain $\omega$ is of \emph{$\ov p$-intersection} if $\omega$ and its coboundary, $\delta \omega$,
are $\ov p$-allowable. 

We denote  $\tN^{*,\ttP}_{\ov{p}}(X)$ 
the complex of  $\ov p$-intersection cochains and
 $ \IH_{\ov p}^{*,\ttP}(X)$  its homology, called the
 \emph{blown-up  intersection cohomology} of $X$ 
 with coefficients in~$\sf R$,
for the perversity $\ov p$. %
\end{definition}

\begin{example}\label{exam:simplicialyPL}
Let $X$ be a filtered space and $\ov{p}$ be a perversity on $X$.
In the sequel, we consider the two following cases of subsets $\ttP\subset \singf X$.
\begin{itemize}
\item[a)] If $\ttP=\singf X$, we recover the complex of blown-up $\ov p$-intersection cochains of $X$ 
and the blown-up  intersection cohomology 
defined in \cite{CST5}. We denote them $\lau \tN * {\ov{p}}X$ and $\lau \IH*{\ov{p}}X$. If there is an ambiguity
with the simplicial situation introduced in the next item, we will call them
\emph{the singular blown-up  intersection complex
and the singular blown-up  intersection cohomology.} 

\item[b)] Let $X=|K|$ be the geometric realization of a   full  filtered simplicial complex $K$.
The family of singular simplices induced by the simplices of $K$ is denoted  $\simp K$. 
 Since $K$ is full, this is a sub $\Delta$-set
of $\singf |K|$ and we can apply the previous process with $\ttP=\simp K$.
We denote $\lau \tN * {\ov p} K$ and $\lau \IH* {\ov{p}} K$
the associated complex and cohomology and call them \emph{the simplicial blown-up  intersection complex
and the simplicial blown-up  intersection cohomology.} 
\end{itemize}

In  \thmref{DKTris}, we prove the existence of an isomorphism between the singular and the simplicial blown-up intersection cohomologies,
$\lau \IH* {\ov{p}} K\cong \lau \IH* {\ov{p}} {|K|}$.
\end{example}

The next result allows the existence of relative blown-up intersection cohomologies.

\begin{proposition}\label{prop:restr3}
Let  $(K,\ov p)$ be a perverse full  filtered simplicial complex.
Then, the  two following restrictions,
\begin{itemize}
\item[a)] $\gamma \colon \lau \tN  *  {\ov p} {|K|}  \to \lau \tN  *  {\ov p} {|\cL(K)|}$, and
\item[b)] $\gamma \colon \lau \tN  *  {\ov p} {K}  \to \lau \tN  *  {\ov p} {\cL(K)}$, 
\end{itemize}
are onto maps.
\end{proposition}

\begin{proof}
Let $\comp K=(a,b)$. When $b=\infty$, we have $\cL(K)=K$ and the result is clear. Suppose now  $b<\infty$, which implies
(see Subsection~\ref{Res})
$\comp \cL(K)<\comp K$.

a) It suffices to prove  that the extension, $\eta \in\lau \tN * {} {|K|}$, by 0 of a cochain 
 $\omega \in \lau \tN*{\ov p} {|\cL(K)|}$  belongs to $\lau \tN * {\ov{p}} {|K|}$.  We clearly have $\|\eta\| \leq \|\omega\|$.
 Now, we want to bound $\|\delta_{|K|}\eta\|$. For that,  we write
 $$\delta_{|K|}\eta=\delta_{|\cL(K)|}\eta+(\delta_{|K|}-\delta_{|\cL(K)|})(\eta)$$
 and notice first that 
 $\|\delta_{|\cL(K)|}\eta\|\leq \|\delta_{|\cL(K)|}\omega\|$.
 As the perverse degree of a linear combination is the maximum of the perverse degrees of its terms  (see \defref{def:thomwhitney}), we are reduced to study 
 $\|(\delta_{|K|}-\delta_{|\cL(K)|})(\eta)\|$. 
 We claim that
 \begin{equation}\label{equa:todosing}
\|(\delta_{|K|}-\delta_{|\cL(K)|})(\eta)\| \leq   \|\omega\|,
 \end{equation}
 which gives the result. 
 Without loss of the generality, it suffices 
 to prove that
$$
\|\delta_{|\cL(K)|} \1_{(F,\varepsilon)}-\delta_{|K|} \1_{(F,\varepsilon)}\| \leq \|\1_{(F,\varepsilon)}\|,
$$
where  $\sigma \colon  \Delta \to |K|$ is a regular singular simplex   and $(F,\varepsilon)$ is a face of  the blown-up $\tDelta$
such that $\sigma(F)\subset |\cL(K)|$.

Since $K_{n-a-1}=\emptyset$ and $\comp\cL(K)<\comp K$, we have $F_0=\dots=F_{n-a-1}=\emptyset$  
and $\sigma(F_{n-a}) \subset |K^{(b-1)}_{n-a}| $.
An element of $\delta_{|\cL(K)|} \1_{(F,\varepsilon)}-\delta_{|K|} \1_{(F,\varepsilon)}$ is of the form
$
\pm\1_{(H,\tau)}
$
and we can suppose, without loss of generality, that $(H,\tau)$ is a face of $\tDelta$. This face verifies
$(F,\varepsilon) \triangleleft  (H,\tau)$,
$\dim (H,\tau) = \dim (F,\varepsilon)+1$, and
$\sigma (H) \not\subset |\cL(K)|$ (i.e., $\sigma( H_{n -a })\not\subset  K^{(b-1)}_{n-a}$).
 This implies  $H_{n-a}\ne F_{n-a}$,  $H_i=F_i$, for $i\ne n-a$, and $\tau =\varepsilon$. Since
$
\|\1_{(H,\tau)}\| = \|1_{(F,\varepsilon)}\|
$
(cf. \remref{Milagro}), we get the claim.

\smallskip

b) The proof of the singular situation is similar. 
\end{proof}

From \propref{prop:restr3}, we deduce the two  following exact sequences,
\begin{equation}\label{sucesasta}
\xymatrix@R=.2cm{
 0 \ar[r] & 
  \lau \tN {*} {\ov p}{|K|,|\cL(K)|} 
 \ar@{^{(}->}[r] & 
   \lau \tN  *  {\ov p}{|K|}  \ar[r]^-\gamma  
 \ar[r] &\lau \tN  *  {\ov p}{|\cL(K)|}  \ar[r] &  0, \\
 0 \ar[r] & 
  \lau \tN {*} {\ov p}{K,\cL(K)} 
 \ar@{^{(}->}[r] & 
   \lau \tN  *  {\ov p}{K}  \ar[r]^-\nu  
 \ar[r] &\lau \tN  *  {\ov p}{\cL(K)}  \ar[r] &  0. \\
 }
 \end{equation}
The homology of $\lau \tN {*} {\ov p}{|K|,|\cL(K)|}$ is isomorphic to the relative blown-up $\ov p$-intersection cohomology
  introduced in \cite{CST5}. We denote it $\crH^*_{\ov{p}}(|K|,\cL(|K|)$. The homology of the complex 
  $  \lau \tN {*} {\ov p}{K,\cL(K)}$ is denoted $  \lau \IH {*} {\ov p}{K,\cL(K)} $.
 In the next statement, we show that the relative intersection cohomology can be decomposed, 
 as does the relative homology of $(X^{(\ell)},X^{(\ell-1)})$ in the case of a CW-complex $X$.

\begin{proposition}\label{RelativoBiss} 
Let  $(K,\ov p)$ be a  perverse full filtered simplicial complex.
The restriction map induces the isomorphisms
\begin{itemize}
\item[i)] $\lau \IH{*}{\ov p}{|K|,|\cL(K)|} 
\xrightarrow{\cong}\displaystyle 
\prod_{\beta \in \cB(K )}   \lau \IH * {\ov p} {|\beta * L_\beta| ,|\cL(\beta * L_\beta)|}$, and
\item[]
\item[ii)] $\lau \IH{*}{\ov p}{K,\cL(K)} 
\xrightarrow{\cong}\displaystyle 
\prod_{\beta \in \cB(K )}   \lau \IH * {\ov p} {\beta * L_\beta ,\cL(\beta * L_\beta)}$.
\end{itemize}
\end{proposition}
\begin{proof}
i) We proceed as in the proof of \propref{SimplRelBis}  keeping the same notations.
 Applying \cite[Proposition 11.3]{CST5}, we get that the inclusion induces the  isomorphism $\lau \IH * {\ov p} { W}\cong \lau \IH * {\ov p} {|\cL(K)|}$ 
which gives an isomorphism
$$
\lau \IH * {\ov p} {|K|, W} 
\cong
\lau \IH * {\ov p} {|K|,|\cL (K)|}.
$$
 By excision (cf. \cite[Proposition 12.9]{CST5}) we have
$$
\lau \IH * {\ov p} {|K|, W} 
\cong
 \lau \IH * {\ov p} {|K| \menos |\cL(K)|, W \menos |\cL(K)|}.
$$
From the decompositions made in \eqref{union}, we get
$$
 \lau \IH * {\ov p} {|K| \menos |\cL(K)|, W \menos |\cL(K)|} 
\cong
 \prod_{\beta \in \cB(K )} \lau \IH * {\ov p} {| \beta* L_\beta| \menos |\partial \beta *L_\beta|, 
 | \beta* L_\beta| \menos (\{ b_\beta \} \cup |\partial  \beta*L_\beta| }). 
$$
Using excision relatively to the closed subset $|\partial\beta  * L_\beta |$, we obtain 
$$
\lau \IH * {\ov p} {|\beta  * L_\beta ) \menos |\partial \beta  *L_\beta |, 
 |\beta  * L_\beta | \menos (\{ b_\beta  \} \cup |\partial \beta  *L_\beta  |)} \cong \lau \IH * {\ov p} {|\beta  * L_\beta |,|\beta  *L_\beta | \menos  \{ b_\beta  \} }. 
$$
Finally, applying \cite[Theorem D]{CST5} to the homeomorphism\eqref{13}, we deduce
$$
\lau \IH * {\ov p} {|\beta  * L_\beta |,|\beta  *L_\beta | \menos  \{ b_\beta  \} } \cong
\lau \IH * {\ov p} {|\beta  * L_\beta |,|\partial \beta  *L_\beta | }
$$
and therefore  the isomorphism i).

\medskip

ii) 
The complex $\lau \tN {*} {\ov p} {K,\cL(K)}$ is made up of the cochains of $K$ vanishing on $\cL (K)$. From \eqref{interLK} we get
  $
  \lau \tN {*} {\ov p} {K,\cL(K)} \cong \prod_{\beta  \in \cB(K )} \lau \tN {*} {\ov p} {\beta *L_\beta ,\cL(\beta  * L_\beta )}
  $
and therefore,
$$
 \lau \IH  *  {\ov p}{K,\cL(K)}  \cong \prod_{\beta \in \cB(K )} \lau \IH  *  {\ov p} {\beta *L_\beta ,\cL(\beta  * L_\beta )},
 $$
by restriction.  
\end{proof}

The next result specifies the  blown-up  intersection cohomology of some pieces of the decomposition brought by \propref{RelativoBiss}.

\begin{proposition}\label{JoinMain} 
Consider a perverse full filtered simplicial complex $(K,\ov p)$.
Let  $\beta \in \cB(K)$ be a clot such that the stratum $Q \in \cS_K$ containing $\beta$ is singular. 
We have
\begin{itemize}
\item [i)]$
\lau \IH *{\ov p} {|\beta * L_\beta|} \cong
\left\{
\begin{array}{ll}
\lau \IH *{\ov p} {|L_\beta|} & \hbox{if } *\leq \ov  p(Q),\\
0 & \hbox{if not,}
\end{array}
\right.
$
\item[]
\item[ii)] 
$
\lau \IH *{\ov p} {\beta * L_\beta} \cong
\left\{
\begin{array}{ll}
\lau \IH *{\ov p} {L_\beta} & \hbox{if } *\leq \ov  p(Q),\\
0 & \hbox{if not,}
\end{array}
\right.
$
\end{itemize}
where the isomorphisms are  induced by the natural inclusion $L_\beta \hookrightarrow \beta * L_\beta$.
\end{proposition}
\begin{proof} If $L_\beta = \emptyset$ then $\lau N *{\ov p} {\beta * L_\beta}  = \lau N *{\ov p} { \beta}  =0$ since there are no regular simplices ($Q$ is singular). So, we can suppose $L_\beta \ne \emptyset$. Let $\comp K = (a,b)$.

i) We proceed as in the proof of \propref{CalculoJoinBis} keeping the same notations.
The case $\dim \beta =0$ is given by \cite[Theorem E]{CST5}, since $L_\beta \ne \emptyset$. 
When $\dim \beta >0$ by using  \cite[Proposition 11.3]{CST5},
 we get that $(F(-,1))^* \colon \lau \IH * {\ov p} {|\beta * L_\beta|} \to \lau \IH * {\ov p} {|\beta * L_\beta|}$ is the identity.
Finally, since $\id   \colon | \eta * L_\beta|  \hookrightarrow | \beta * L_\beta| \xrightarrow{\xi} | \eta *L_\beta|$ and 
$F(-,1) \colon |\beta * L_\beta|    \xrightarrow{\xi} | \eta *L_\beta|  \hookrightarrow |\beta * L_\beta|$, 
we conclude that the inclusion $ | \eta * L_\beta|  \hookrightarrow |\beta * L_\beta|$ 
induces the isomorphism $ \lau\IH *{\ov p} {|\beta * L_\beta|}  \cong \lau \IH *{\ov p} {| \eta*L_\beta|} $.

ii) 
By definition, we have
$$
\Hiru \tN  *  {\beta * L_\beta} =  \Hiru N * {\tc \beta } \otimes \Hiru \tN  *  {L_\beta},
$$
and any cochain $\omega \in \lau \tN  * {} {\beta * L_\beta} $
can be written as a sum 
$$ 
\omega = \sum_{F \triangleleft  \beta}  \1_{(F,0)} \otimes \omega_F + \sum_{F\triangleleft  \beta} 
\1_{(F,1)}  \otimes \tau_F +    \1_{(\emptyset ,1)}  \otimes \tau_\emptyset,
$$
with $\omega_{F},\,\tau_{F},\,\tau_{\emptyset}$ in $\Hiru \tN *  {L_\beta}$
if $F \triangleleft  \beta$.
 The perverse degrees are computed as follows,
\begin{itemize}
\item[-] $\|\1_{(F,0)} \otimes \tau\|_\ell
 = 
\left\{
\begin{array}{ll}
\deg \tau & \hbox{if } \ell= a, \\
\|\tau\|_\ell &\hbox{if } \ell<a,
\end{array}
\right.
$

\item[]

\item[-] $\| \1_{(F,1)} \otimes \tau\|_\ell = 
\left\{
\begin{array}{ll}
-\infty & \hbox{if } \ell = a,\\
\|\tau\|_\ell &\hbox{if } \ell <a,
\end{array}
\right.$
\end{itemize}
for each $\ell \in \{1, \ldots,n\}$.  The condition $\|\omega\|\leq \ov p$ is equivalent to
$\max(\|\omega_F\|, \|\tau_F\|,\|\tau_\emptyset\|) \leq \ov p$ and $\deg \omega_F  \leq \ov  p(Q)$, for each $F \triangleleft  \beta$.
 For $\|\delta \omega\| \leq \ov p$, this becomes
 $\max(\|\delta\omega_F\|, \|\delta\tau_F\|,\|\delta\tau_\emptyset\|) \leq \ov p$, and $\deg (\delta\omega_F )\leq \ov  p(Q)$.
In particular, we have, for each $F\triangleleft \beta$,
$$\omega_F \in 
\lau \tN{<\ov p(Q)} {\ov p }  {L_\beta}  \oplus \left( \lau \tN{\ov  p(Q)}{\ov p }  {L_\beta} \cap \delta^{-1}(0)\right)
\text{ and } \tau_F , \tau_\emptyset \in \lau \tN*{\ov p} {L_\beta}.$$ 
The complex   $\Hiru N * {\tc\beta, \beta} $
 is made up of the cochains on $\tc \beta$ vanishing on $\beta$.
 It is generated by the family $\{ \1_{(F,1)} \mid F =\emptyset \hbox{ or } F \triangleleft \beta\}$.
Let us consider the two 
short
 exact sequences,
$$
0\to \Hiru N * {\tc\beta, \beta} \hookrightarrow \Hiru N *  {\tc \beta} \xrightarrow{\nu} \Hiru N * {\beta} \to 0,
$$
where 
$
\nu \left( 
 \sum_{F \triangleleft\beta} n_F \,  \1_{(F,0)}
 + \sum_{F\triangleleft \beta} m_F \,  \1_{(F,1)}
+ n_\emptyset \, \1_{(\emptyset,1)} 
\right) 
= \sum_{F \triangleleft \beta}  u_F \, \1_{F}, 
$
with $n_{F},\,m_{F},\,n_{\emptyset}\in \sf R$,
and
$$
0\to  \Hiru N * {\tc\beta, \beta}  \otimes \lau \tN  * {\ov p} {L_\beta} 
\hookrightarrow 
 \lau \tN*{\ov p} {\beta *L_\beta}
\xrightarrow{\upsilon} 
 \Hiru N * {\beta}  \otimes \tau_{\ov  p(Q)}\lau \tN * {\ov p} {L_\beta} \to 0,
$$
where 
$
\upsilon\left( 
 \sum_{F \triangleleft  \beta} \1_{(F,0)} \otimes \omega_F 
 + \sum_{F\triangleleft  \beta}  \1_{(F,1)}  \otimes \tau_F 
+   \1_{(\emptyset,1)} \otimes \tau_{\emptyset}
\right)
 =    \sum_{F \triangleleft  \beta} \1_{(F,0)} \otimes \omega_F. 
$

Since the map $\nu$ is a quasi-isomorphism,  the complex $ \Hiru N * {\tc\beta, \beta} $ is acyclic. 
So, the map $\upsilon$ is a quasi-isomorphism. 
The result comes from  $\Hiru H j {\beta} = 0$ if $i>0$ and $\Hiru H 0 {\beta} =\sf R$.
\end{proof}

When $Q$ is a regular stratum, we have $L_\beta = \emptyset$, 
$\lau \IH i {\ov p}  {|\beta *L_\beta|} =\lau \IH i {\ov p} {|\beta|}  =  \Hiru H i{| \beta |}= 0$ if $i>0$ and $\Hiru H 0 {| \beta |}=\sf R$. 
Similarly, the simplicial blown-up intersection cohomology verifies
$\lau \IH i {\ov p}  {\beta *L_\beta} =\lau \IH  i {\ov p}  {\beta}  =  \Hiru H i{\beta }=0$ if $i>0$ and $\lau \IH  0 {\ov p}  {\beta}=\Hiru H 0 { \beta }=\sf R$.

\medskip
The second main result of this work establishes an isomorphism between the singular and the simplicial blown-up intersection cohomologies.
These cohomologies are related by the cochain map
$\rho \colon
\lau \tN * {\ov p} {|K|} \to \lau \tN * {\ov p} K
$
induced by  the natural inclusion $\simp K \subset \singf |K|$ (cf. \eqref{iota}).

\begin{theorem}\label{DKTris}
Let $(K,\ov p)$ be a perverse full  filtered simplicial complex. Then, the map $\rho$ induces  an isomorphism
\begin{equation}\label{B}
\lau \IH {*} {\ov p} {|K|} {\cong}   \lau \IH  *  {\ov p} {K}.
\end{equation}
\end{theorem}

\begin{proof}
We proceed in two steps.

\medskip
\emph{First Step :} Suppose  $\comp K=(a,b)$ with $b < \infty$.  
Given a subcomplex $L \subset K$, we also have   $b_L\leq  \dim L< \infty$, with $\comp L = (a_L,b_L)$, and therefore $\comp \cL(L) < \comp L$ (cf. Subsection~\ref{Res}).
We use an induction on the complexity $(a,b)$ of $K$.
When $a=b=0$, the complex $K$ is a discrete family of 0-dimensional  simplices, so we have 
$\lau N* {\ov p} K = \lau N *{\ov p}  {|K|}$ and  the claim.
For the  inductive step, we consider the following commutative diagram deduced from \eqref{sucesasta},
$$
\xymatrix{
0 \ar[r] & 
\lau N  * {\ov p} {|K|,|\cL(K)|}
\ar[d]  
\ar@{^{(}->}[r]
&
\lau N *{\ov p}  {|K|} 
  \ar[r]^-\gamma \ar[d]  &\lau N* {\ov p}  { |\cL(K)|}  \ar[d] \ar[r]& 0\\
0 \ar[r] & 
 \lau N  * {\ov p} {K,\cL(K)}
 \ar@{^{(}->}[r]&
\lau N * {\ov p}  K
   \ar[r]^-{\gamma} 
   &\lau N  * {\ov p} {\cL(K)}
    \ar[r]& 0
}
$$
where the vertical maps are induced by the map $\rho$.
From the induction hypothesis, 
we know that the right arrow is a quasi-isomorphism.
So, it suffices to prove that the 
left  arrow is a quasi-isomorphism. Using \propref{RelativoBiss},  this assertion
is equivalent to the fact that
the  map $\rho$ induces a quasi-isomorphism, for each  $\beta \in \cB(K )$,
$$
  \lau N *{\ov p}  {|\beta * L_\beta| , |\cL(\beta * L_\beta|)}
  \to
  \lau N  *{\ov p} {\beta * L_\beta ,\cL(\beta * L_\beta}).
$$
From \eqref{sucesasta},  we deduce
the commutative diagram 
$$
\xymatrix{
0 \ar[r] & 
\lau N  * {\ov p} {|\beta * L_\beta|,|\cL(\beta * L_\beta)|}
\ar[d]  
\ar@{^{(}->}[r]
&
\lau N * {\ov p}  {|\beta * L_\beta|} 
  \ar[r]^-\gamma \ar[d]  &\lau N {\ov p} * { |\cL(\beta * L_\beta)|}  \ar[d] \ar[r]& 0\\
0 \ar[r] & 
 \lau N *  {\ov p} {\beta * L_\beta,\cL(\beta * L_\beta)}
 \ar@{^{(}->}[r]&
\lau N * {\ov p}  {\beta * L_\beta}
   \ar[r]^-{\gamma} 
   &\lau N *  {\ov p} {\cL(\beta * L_\beta)}
    \ar[r]& 0
}
$$
where the vertical maps are induced by $\rho$.
We have  
$
 \comp \cL(\beta * L_\beta) < \comp (\beta * L_\beta) \leq \comp K
$ (cf. Subsection~\ref{IFC})
and the induction hypothesis implies  
that the  right arrow is a quasi-isomorphism. It remains to prove that the middle arrow is a quasi-isomorphism. 
We distinguish two cases.

\medskip
$\bullet$ \emph{$\beta$ is included in a regular stratum}. We have $L_\beta=\emptyset$ and
this case is resolved in the paragraph following the proof of \propref{JoinMain}.

\medskip
$\bullet$ \emph{$\beta$ is included in a singular stratum}. 
  We have $\comp L_\beta < \comp K$ (cf. Subsection~\ref{links}).
From \propref{JoinMain} and   the induction hypothesis, the 
middle arrow is a quasi-isomorphism.

\medskip
\emph{Second Step : Suppose  $\comp K=(a,\infty)$}. 
We proceed by induction on $a \in \{0, \ldots,n\}$. When $a=0$,  $K$  has no singular part, that is, $K_{n-1}=\emptyset$. 
So, from the classical situation, the  map $\rho$ induces the isomorphism $\lau H {*} {} {|K|} {\cong}   \lau H  *  {} K$. 
Let us  consider the inductive step with $a>0$.

Given $k\in \N$, we define 
$
K^{k} = \{ \sigma \in K \mid \dim \sigma_{n-a} \leq k\}
$
and $K^{-1} = \{\sigma \in K \mid \sigma_{n-a}=\emptyset\}$.
They are simplicial subcomplexes of $K$ with $K = \bigcup_{k \geq -1} K^k$ and 
there exists an  infinite sequence
$$
K^{-1} \subset K^0 \subset \dots\subset  K^k \subset K^{k+1} \subset \dots 
$$
Each of theses complexes is endowed with the induced structure defined in Subsection~\ref{IFC}.
We prove that the filtered simplicial complex $K^k$, with $k \in \N \cup \{-1\}$, verifies \eqref{B}.

Let us begin with the case  $k=-1$. By construction, we have $\comp (K^{-1}) = (a_-,b_-) < (a,0)$.  
If $b_-< \infty$, the First Step gives the claim. If $b_-  = \infty$,  the inductive step assures us that $K^{-1}$ verifies \eqref{B}. 
Let $k\in \N$ with $K^k \ne K^{k-1}$, we have $\comp K^{k} = (a,k)$. 
Following the First Step, we conclude that $K^k$ verifies \eqref{B}. 
We also have $\cL(K^k) = K^{k-1}$ by definition. By using \propref{prop:restr3}, we know  that the morphisms,
\begin{equation}\label{rhoand}
\rho^* \colon \lau N * {\ov p} { K^{k}} \to \lau N * {\ov p}  {K^{k-1}} \hspace{1cm} \hbox{ and }  \hspace{1cm}   \rho^* \colon
\lau N* {\ov p} { |K^{k} |} \to \lau N * {\ov p}  { |K^{k-1}|},
\end{equation}
 induced by the inclusion $K^{k-1} \hookrightarrow K^{k}$, are onto maps.

Associated to the  directed set $K = \bigcup_{k\geq -1} K^{k} = \varinjlim K^{k}$,
 we have the towers $(\lau N * {\ov p} {K^{k} } )_{k\geq -1}$ and $(\lau N* {\ov p} { |K^{k} | })_{k\geq -1}$. 
 As they are Mittag-Leffler  (see \eqref{rhoand}), we have the commutative diagram,
$$
\xymatrix{
0\ar[r]
&{\varprojlim}^1\lau \IH {*-1}{\ov p}{ | K^{k} |}\ar[r] \ar[d]
&\lau \IH * {\ov p} {|K|} \ar[r]\ar[d]
&\varprojlim \lau \IH {*}{\ov p}{ | K^{k}| }\ar[r]\ar[d]
&0 \\
0\ar[r]
&{\varprojlim}^1\lau \IH {*-1}{\ov p}{K^{k} }\ar[r]
&\lau \IH * {\ov p} K \ar[r]
&\varprojlim \lau \IH {*}{\ov p}{K^{k} }\ar[r]
&0,
}
$$
where the vertical maps are induced by $\rho$. From the first step,  the left and right maps are isomorphisms.
So the Five's Lemma ends the proof. \end{proof}

\section{Simplicial and singular versus PL}

\begin{quote}
Let $X$ be a PL space relatively to a family  $\cT$ of triangulations, endowed with a PL filtration.
In this section, we first define the PL blown-up intersection cohomology of $X$. 
In \thmref{TeoremaPL}, we prove that this cohomology is isomorphic 
to the singular blown-up intersection cohomology of $X$
and to the simplicial blown-up cohomology of any full triangulation belonging to $\cT$.
\end{quote}

Let us begin with basic recalls on  PL spaces.

\begin{definition}\label{def:triangulation}
A \emph{triangulation} of $X$ is a pair $(K,h)$ of a   simplicial complex, $K$, and a homeomorphism,
$h\colon |K|\to X$. A \emph{subdivision} of $(K,h)$ is a pair $(L,h)$ where $L$ is a subdivision of $K$,  denoted $L\vartriangleleft K$.
Two triangulations, $(K,h)$ and $(L,f)$, of $X$ are \emph{equivalent} if $f^{-1} h\colon |K|\to X\to |L|$
is induced by a simplicial isomorphism $K\to L$.
\end{definition}

\begin{definition}\label{df:plspace}
A \emph{PL space} is a second countable, Hausdorff topological space, $X$, endowed with a family, $\cT$, of triangulations of $X$, 
called admissible and such that
the following properties are satisfied.
\begin{enumerate}[(a)]
\item If $(K,h)$ in $\cT$ and $L\vartriangleleft K$, then $(L,h)\in\cT$.
\item If $(K,h)\in\cT$ and $(L,f)\in \cT$, they have  a common subdivision in $\cT$.
\end{enumerate}
\end{definition}

\begin{definition}\label{def:plmap}
Let $(X,\cT)$ and $(Y,\cS)$ be two PL spaces.
A \emph{PL map}, $\psi\colon (X,\cT)\to (Y,\cS)$, is a continuous map,
$\psi\colon X\to Y$, such that for any $(K,h)\in\cT$
and any $(L,f)\in\cS$, there is a subdivision $K'$ of $K$ for which $j^{-1}\psi h$ takes each simplex of $K'$ linearly \emph{into} a simplex of $L$.
A \emph{PL subspace} of $(X,\cT)$ is a PL space $(X',\cT')$, such that $X'$ is a subspace of $X$ and the inclusion map,
$X'\hookrightarrow X$, is a PL map.
\end{definition}

\begin{definition}\label{de:PLfilter}
A \emph{filtered PL space} is a PL space $(X,\cT)$, filtered by a sequence of closed PL subspaces,
$$
X=X_{n}\supsetneq X_{n-1}\supset \dots\supset X_{0}\supset X_{-1}=\emptyset.
$$
\end{definition}
From \cite[Subsection 2.5.2]{LibroGreg}, we may suppose that 
there is a triangulation $(K,h)$ of $X$ with respect to which each of the $X_i$ 
is the image under $h$ of a subcomplex of $K$.

\medskip
If $L$ is a full subcomplex of a simplicial complex $K$ and $K'$  is a subdivision of $K$, we denote $L'$ the subdivision of $L$
induced by $K'$. If $L$ is full in $K$, then (\cite[Lemma 3.3]{MR665919}) $L'$ is full in $K'$.
Therefore the fullness property  adapts to PL spaces: any filtered PL space admits an admissible full triangulation,
see \cite[Lemma 3.3.19]{LibroGreg}.

 \subsection{Blown-up and subdivision}
 Let us connect the blown-up cochain complex of a   simplicial complex and of  one of its subdivisions.
 
\begin{proposition}\label{prop:Ksub}
Let $K'$ be a subdivision of a full filtered   simplicial complex $K$ and $\ov{p}$ be a perversity.  
There exist  cochain maps
 $$j_{K'K}\colon \lau{\tN} *{\ov p}{K'}\to \lau{\tN} *{\ov p}K
\quad  \text{and} \quad
 \varphi_{KK'}\colon  \lau{\tN} *{\ov p}K\to  \lau{\tN} *{\ov p}{K'}$$
 such that $j_{K'K}\circ \varphi_{KK'}=\id$.
 \end{proposition}
 
 \begin{proof}
Let $\sigma\colon \Delta_{\sigma}\to K$ be an oriented, regular, filtered simplex of $K$. The
filtration on $K$ induces a decomposition in join product,
$\Delta_{\sigma}=\Delta_{\sigma,0}\ast\dots\ast\Delta_{\sigma,n}$,
where each $\Delta_{\sigma,i}\menos \Delta_{\sigma,i-1}$ is included in a stratum $S_{i}$. 
The simplex $\sigma$ is subdivided in some oriented simplices 
 $(\sigma\{\ell\})_{1\leq \ell \leq k}$ of $K'$, of equal dimension. 
 Let $\ell\in\{1,\dots,k\}$.
 We can suppose that the orientations of $\sigma$ and $\sigma\{\ell\}$ are compatible.
 Each simplex of the subdivision is obtained by adding new vertices to a subset of $\vert(K)$.
 These new vertices belongs to a stratum and we can write  $\Delta_{\sigma\{\ell\}}$
 as a join product
 $\Delta_{\sigma\{\ell\}}=\Delta_{\sigma\{\ell\},0}\ast\dots\ast \Delta_{\sigma\{\ell\},n}$.

$\bullet$  \emph{Construction of $j_{K'K}$.}
For any $\varepsilon$ as in  \defref{43}
and any $\ell\in\{1,\dots,k\}$, we set
$$j\left( \1_{(\Delta_{\sigma\{\ell\}},\varepsilon)}\right)=
\1_{(\Delta_{\sigma},\varepsilon)}.$$
We get a cochain map $j\colon \Hiru {\tN} *{K'}\to \Hiru{ \tN} *K$. Concerning the perverse degrees, let us  notice that
$\dim( \Delta_{\sigma\{\ell\}} \cap S)\leq \dim (\Delta_{\sigma}\cap S)$.
Thus, we have $\|\1_{(\Delta_{\sigma},\varepsilon)}\|\leq \|\1_{(\Delta_{\sigma\{\ell\}},\varepsilon)}\|$
for any $\ell\in\{1,\dots,k\}$.
Defined on a basis, the association $j$ extends linearly in a  map
 $$j_{K'K}\colon \lau{\tN} *{\ov p}{K'}\to \lau{\tN} *{\ov p}K.$$
On each factor of the join product, this map is the transposed map of the subdivision map. 
From \cite[Lemmas 3.3.1, 3.3.15]{LibroGreg}, it follows that $j_{K'K}$ is a chain map.

 \medskip
$\bullet$  \emph{Construction of $\varphi_{KK'}$.}
 We use a simplicial map, 
 $\nu\colon K'\to K$,
built by Goresky and MacPherson in \cite[Appendix]{MR833195}.
Let us recall their construction, using  \cite[Subsection  3.3.4]{LibroGreg}, assuming that $K$ is full
and that the set  $\vert(K)$ is well ordered.
The map $\nu$ is defined from a map $\ov{\nu}\colon \vert (K') \to \vert (K)$ between the set of vertices of $K$ and $K'$,
that is  extended linearly on simplices.

If the vertex $v'\in K'$ is already in $K$, we set $\ov{\nu}(v')=v'$.
Otherwise, $v'$ is in the interior of a simplex $\sigma$ of $K$. 
Denote  $S$ the stratum of $|K|$ containing the interior of $\sigma$.
One knows (see \cite{MR833195} or \cite[Lemma 3.3.25]{LibroGreg}) that the interior of a simplex
is contained in the stratum $S$ if, and only if, all the vertices of $\sigma$ are contained in the closure $\ov{S}$
and at least one vertex of $\sigma$ is in $S$.
We define $\ov{\nu}(v')$ as the vertex of $\sigma$ in $S$ that is greatest in the selected order.
As the vertices $v'$ and $\ov{\nu}(v')$ are in the same stratum (see \cite[Proof of Lemma 3.3.21]{LibroGreg}),  the map $\ov{\nu}$ is compatible with the strata decomposition.
Also, from the same proof, it is explicit that \emph{only one} 
of the simplices $\Delta_{\sigma\{\ell\}}$ of the subdivision has an image by $\nu$ which is
of the same dimension than $\Delta_{\sigma}$. Let us denote it $\Delta_{\sigma\{\nu\}}$ and set
$$\varphi_{KK'}\left(\1_{(\Delta_{\sigma},\varepsilon)}\right)
=
 \1_{(\Delta_{\sigma\{\nu\}},\varepsilon)}.
$$
As $\nu\colon K'\to K$ is a simplicial map, compatible with the strata, we have
 $\|\1_{(\Delta_{\sigma},\varepsilon)}\|= \|\1_{(\Delta_{\sigma\{\nu\}},\varepsilon)}\|$ and a  map
 $$
  \varphi_{KK'}\colon  \lau{\tN} *{\ov p}K\to  \lau{\tN} *{\ov p}{K'}.
 $$
 As the association $\ov{\nu}$ gives a chain map (see \cite[Appendix]{MR833195} or \cite[Lemma 3.3.21]{LibroGreg}), 
 by duality, $\varphi_{KK'}$ is a chain map.
The equality $j_{K'K}\circ \varphi_{KK'}=\id$ follows directly from the definitions of the two maps.
 \end{proof}

\subsection{PL blown-up cohomology and simplicial cohomology}

 Let $(X,\cT)$ be a PL filtered space and $K$ be any triangulation of $\cT$. We denote  $\sd\; K$
 the barycentric subdivision of $K$ and  $(\sd^{i}\;K)_{i\in \N}$ the family of iterated barycentric subdivisions,
 with the convention $\sd^0=\id$.
 The subdivision maps give a direct system
 $\sub_{i}\colon \sd^{i+1}\; K\to \sd^i\; K$, for $i\in\N$.
Let $\ov{p}$ be a perversity on $X$.
In \cite[Lemma 5.4.1]{LibroGreg}, Friedman proves that the PL homology is obtained as the homology of the inductive limit
of this direct system, 
$\lau H {\ov{p}} {*,\PL}X \cong\hiru H {*}{\varinjlim_{i} \lau C {\ov{p}} * {\sd^i\;K}}$.

\medskip
Denote 
$j_{i}\colon \lau{\tN} *{\ov p}{\sd^{i}\; K}\to \lau{\tN} *{\ov p}{\sd^{i-1}\;K}$
and
$\varphi_{i}\colon \lau{\tN} *{\ov p}K\to \lau{\tN} *{\ov p}{\sd^i\;K}$
from the maps of \propref{prop:Ksub}, for $i\in\N$.
More specifically, we set $j_{i}=j_{\sd^iK\ \sd^{i-1}K}$ and, by induction,
$\varphi_{i}=\varphi_{\sd^{i-1}K\ \sd^iK}\circ \varphi_{i-1}$ and $\varphi_{0}=\id$. 
By construction and \propref{prop:Ksub}, it follows $j_{i}\circ \varphi_{i}=\varphi_{i-1}$. 
The maps $j_{i}$ define a projective system which allows the following definition.

 \begin{definition}\label{def:plcochain}
 Let $(X,\cT)$ be a PL filtered space and $\ov{p}$ be a perversity on $X$. The complex
 of  $\ov{p}$-intersection PL cochains is the  inverse (projective) limit,
 $$
 \lau \tN {*,\PL} {\ov{p}}X=\varprojlim_{i} \lau {\tN} *{\ov p}{\sd^i\;K}.
 $$
 We denote $\lau \IH {*,\PL} {\ov{p}} X $ the corresponding cohomology, and call it
 the \emph{$\PL$ blown-up $\ov{p}$-intersection  cohomology.}
 \end{definition}

 From \lemref{bary}, we can suppose that $K$ is  full.
 Let's now prove that the PL blown-up $\ov{p}$-intersection cohomology is isomorphic 
 to the blown-up  singular and the blown-up  simplicial ones, and thus is independent of the choice of $K$.
 
 \begin{theorem}\label{TeoremaPL}
 Let $(X,\cT)$ be a PL filtered space and $\ov{p}$ be a perversity on $X$. Then for a pure $K\in\cT$,
 there are isomorphisms,
 $$\lau \IH {*,\PL} {\ov{p}}X
 \cong 
 \lau \IH {*} {\ov{p}}K
 \cong 
\lau \IH {*} {\ov{p}}X
\text{ and }
\lau H {\ov{p}} {*,\PL}X
\cong
\lau H {\ov{p}} * K
\cong
\lau H {\ov{p}} * X.
 $$
 \end{theorem}
 
 The last part recovers \cite[Theorem 5.4.2]{LibroGreg} without the hypothesis of CS set structure.

 \begin{proof}
 Let $j_{\leq i}$ be the map defined by $j_{\leq 1}=j_{1}$
 and
 $j_{\leq i}=j_{\leq i-1}\circ j_{i}$.
 We also denote $p_{i}\colon  \lau \tN {*,\PL}{\ov{p}}X\to   \lau {\tN}* {\ov p} { \sd^{i}(K)}$
 the projection given by the projective limit.
  \begin{equation}\label{equ:inversePLtodo}
 \xymatrix{
 &&
 \lau{\tN} *{\ov p}K\\
 \lau \tN{*,\PL} {\ov{p}}X
 \ar[rru]^-{p_{0}}
 \ar[rr]_-{p_{i-1}}
 \ar[rrd]_-{p_{i}}
 &&\lau {\tN} *{\ov p}{sd^{i-1}\;K} \ar[u]^-{j_{\leq i-1}}
 &&
 \lau{\tN} *{\ov p}K
 \ar@{=}[llu]_-{\id}
 \ar[ll]^-{\varphi_{i-1}}
 \ar[lld]^-{\varphi_{i}}
 \ar@/^6pc/@{-->}[llll]_-{\Psi}
 \\
 &&
\lau{\tN} *{\ov p}{\sd^i\;K}
\ar[u]^-{j_{i}}
 }
 \end{equation}
 From the equalities $j_{i}\circ \varphi_{i}=\varphi_{i-1}$ and
 the universal property of inverse limits, we get a cochain map $
 \Psi\colon \lau{\tN} *{\ov p}K\to  \lau \tN {*,\PL} {\ov{p}}X$
 such that 
 $p_{0}\circ \Psi=\id$
 and
 $p_{i}\circ\Psi= \varphi_{i}$
 for any $i\in\N$.
 From \propref{prop:Ksub} applied to $\sd^i \;K$ and $\sd^{i-1}\;K$, we deduce that each $j_{i}$ and each $\hiru H{*}{j_{i}}$
 is surjective.
 Thus \cite[Proposition 13.2.3]{MR217085} and \thmref{DKTris} give the existence of isomorphisms,
$$
\lau \IH {*,\PL} {\ov{p}}X= \hiru H{*}{\varprojlim_{i} \lau{\tN} *{\ov p}{\sd^i\;K}}
\cong \varprojlim_{i}  \lau \IH *{\ov{p}}{\sd^i\;K}\cong  \lau \IH * {\ov{p}}K
\cong  \lau \IH * {\ov{p}}X.
$$
The existence of the isomorphisms in $\ov{p}$-intersection homology follows directly from \thmref{DKBis} and the
commutativity of homology with inductive limits.
 \end{proof}


\end{document}